\documentclass[11pt,a4paper]{article}
\usepackage{amsmath,amsthm}
\usepackage{amsfonts}
\usepackage{hyperref}
\usepackage[capitalize]{cleveref}
\usepackage{stmaryrd}

\usepackage[margin=3cm]{geometry}
\def\ftoday{le \space\number\day \space\ifcase\month\or
  janvier\or f\'evrier\or mars\or avril\or mai\or juin\or
  juillet\or ao\^ut\or septembre\or octobre\or novembre\or d\'ecembre\fi
  \space\number\year}

\newtheorem{theorem}{Theorem}

\newtheorem{corollary}[theorem]{Corollary}

\newtheorem{definition}[theorem]{Definition}

\newtheorem{lemma}[theorem]{Lemma}

\newtheorem{proposition}[theorem]{Proposition}
\theoremstyle{definition}
\newtheorem{remark}[theorem]{Remark}

\input amssym.def
\input amssym
\newcommand{\re}[1]{(\ref{#1})}
\newcommand{\rl}[1]{Lemma~\ref{#1}}

\newcommand{\crb}{\mathcal{V}}

\newcommand{\inp}[2]{\left\langle #1 , #2 \right\rangle}

\newcommand{\vnorm}[1]{\left\|  #1 \right\|}

\newcommand{\Cn}{\mathbb{C}^n}

\newcommand{\Cd}{\mathbb{C}^d}

\newcommand{\C}{\mathbb{C}}
\newcommand{\CN}{\mathbb{C}^N}

\newcommand{\Rd}{\mathbb{R}^d}
\newcommand{\R}{\mathbb{R}}

\newcommand{\N}{\mathbb{N}}

\newcommand{\dop}[1]{\frac{\partial}{\partial #1}}

\newcommand{\vardop}[3]{\frac{\partial^{|#3|} #1}{\partial {#2}^{#3}}}

\newcommand{\fps}[1]{\C\llbracket#1\rrbracket}
\newcommand{\fpstwo}[2]{#1\llbracket #2\rrbracket}
\newcommand{\cps}[1]{\C\{#1\}}

\newcommand{\todo}[1]{}

\newlength{\extendaxesby}

\DeclareMathOperator{\ord}{ord}
\DeclareMathOperator{\id}{id}

\DeclareMathOperator{\imag}{Im}
\DeclareMathOperator{\real}{Re}

\DeclareMathOperator{\image}{image}

\DeclareMathOperator{\Sym}{Sym}

\def\cF{\mathcal{F}}

\def\cS{\mathcal{S}}
\def\cT{\mathcal{T}}

\def \ank {\mathbb A_n^k}

\usepackage{faktor}
\begin{document}

\title{Convergence of the Chern-Moser-Beloshapka normal forms }
\author{Bernhard Lamel\thanks{%
University of Wien E-mail : \texttt{bernhard.lamel@univie.ac.at}} \hspace{%
1mm}, and Laurent Stolovitch
\thanks{%
CNRS and Universit\'e C\^{o}te d'Azur, CNRS, LJAD,Parc Valrose
06108 Nice Cedex 02, France. E-mail : \texttt{%
stolo@unice.fr}.%
Research of B. Lamel and L.Stolovitch was supported by ANR grant "ANR-14-CE34-0002-01" and the FWF (Austrian Science Foundation) grant I1776 for the international cooperation project "Dynamics and CR geometry". Research of Bernhard Lamel was also supported by the Qatar Science Foundation, NPRP 7-511-1-098.} }
\maketitle

\begin{abstract}
In this article, we first describe a 
normal form of real-analytic, Levi-nondegenerate submanifolds of $\CN$ of codimension $d\geq 1$ under the action of formal biholomorphisms, that is, of perturbations of Levi-nondegenerate hyperquadrics. We give a sufficient condition on the formal normal form that ensures that the normalizing transformation to this normal form is holomorphic. We show that our techniques can be adapted in 
the case $d=1$ in order to obtain a new and direct proof of Chern-Moser normal form theorem.
\end{abstract}

\section{Introduction}

In this paper, we study normal forms for real-analytic, Levi-nondegenerate
manifolds of $\CN$. 
A real submanifold $M\subset \CN$ (of real codimension $d$) is given, 
locally at a point $p\in M$, in suitable coordinates $(z,w)\in\Cn\times\Cd = \CN$, by a defining function of the form 
\[ \imag w = \varphi(z,\bar z, \real w), \]
where $\varphi \colon \Cn\times \Rd\to \Rd$ is a germ of a real analytic
map satisfying $\varphi(0,0,0)=0$, and $\nabla \varphi (0,0,0) = 0$. Its 
natural second order invariant is its Levi form $\mathcal{L}_p$: 
This  is a natural Hermitian vector-valued form,
 defined on $\mathcal{V}_p =\C T_p M \cap \C T^{(0,1)}_p \CN$ as 
\[ \mathcal{L}_p (X_p ,Y_p) = [X_p , \bar Y_p] \mod \crb_p \oplus \bar \crb_p 
\quad \in \faktor{\C T_p M}{\crb_p \oplus \bar \crb_p}.  \] We say that $M$ is Levi-nondegenerate (at $p$) if the Levi-form 
$\mathcal{L}_p$ is a nondegenerate, vector-valued Hermitian form, 
and is of full rank. 

Let us recall that we say that $\mathcal{L}_p$ is {\em nondegenerate} if it satisfies that 
$\mathcal{L}_p (X_p, Y_p) = 0$ for all $Y_p\in \crb_p$ implies $X_p = 0 $
and that we say that $\mathcal{L}_p$ is of full rank, if
 $\theta (\mathcal{L}_p (X_p , Y_p)) = 0$ for all 
$X_p, Y_p \in \crb_p$ and for 
$\theta \in T^0_p M = \crb_p^\perp \cap \bar \crb_p^\perp$ (where $\crb_p^\perp \subset \C T^*M$ is the 
holomorphic cotangent bundle) implies $\theta = 0$.  

The typical model for this situation is a {\em hyperquadric},
that is, a manifold of the form 
\[ \imag w = Q(z,\bar z) = \begin{pmatrix}
   Q_1 (z,\bar z) \\ \vdots \\  Q_d (z,\bar z)
\end{pmatrix} = \begin{pmatrix}
  \bar z^t J_1 z \\ \vdots \\ \bar z^t J_d z
\end{pmatrix}, \]
where each $J_k$ is a Hermitian $n\times n$ matrix, and the conditions of nondegeneracy 
and full rank are expressed by 
\begin{equation}\label{e:nondegenerate} \bigcap_{k=1}^d \ker J_k = \{0\}, \quad \sum_{k=1}^d \lambda_k J_k = 0 \Rightarrow \lambda_k = 0, \quad k =1,\dots, d. \end{equation}
The defining equation of the hyperquadric becomes {\em quasihomogeneous} of degree 1, if 
we endow $z$ with the weight $1$ and $w$ with the weight $2$, 
which we shall do from now on. A Levi-nondegenerate manifold can 
thus, at each point, be thought of as a ``{higher order }deformation'' of a hyperquadric, that is, 
their defining functions $\imag w = \varphi(z, \bar z)$ can be 
rewritten as 
\[ \imag w = Q (z,\bar z) + \Phi_{\geq 3} (z,\bar z, \real w),\]
where $\Phi_{\geq 3}$ only contains quasihomogeneous terms of order 
at least $3$.

We are going to classify  germs of such real analytic manifolds under
the action of the group of germs of biholomorphisms of $\CN$.
The classification problem for Levi-nondegenerate manifolds
has a long history, especially in the case of hypersurfaces ($d=1$). 
It was first studied (and solved) for hypersurfaces 
in $\C^2$ by Elie Cartan in a series 
 of papers \cite{Cartan:1933ux,Cartan:1932ws} in the early 
1930s, using his theory of moving frames. Later on, 
Tanaka \cite{Tanaka:1962ti} and Chern and Moser \cite{Chern:1974wu}
solved the problem for Levi-nondegenerate hypersurfaces in $\C^n$. 
They used differential-geometric approaches, but also, 
in the case of Chern-Moser an approach coming from 
the theory of dynamical systems: finding a normal form for 
the defining function, or equivalently, finding a special coordinate
system for the manifold. We refer to the papers by Vitushkin
\cite{MR786085,MR850493}, the book by Jacobowitz \cite{MR1067341}, the survey by Huang \cite{Huang-survey} and the survey by Beals, Fefferman, and Grossman \cite{Beals:1983kd} 
in which the geometric and analytic significance 
of the Chern-Moser normal form are discussed.

Our paper takes up a very classical 
problem with a new tool, and gives 
a formal normal form for Levi-nondegenerat
real analytic manifolds which 
under a rather simple condition (see 
\eqref{e:crucial})  can be shown to be
convergent. Recent advances in normal 
forms for real submanifolds of 
complex spaces with respect 
to holomorphic transformations have 
been significant: We would like to cite
in this context the recent works of  Huang and Yin \cite{MR2501295,MR3498915,MR3600082}, the second author and Gong
\cite{MR3570294}, and
Gong and Lebl \cite{MR3336931}. 

We will discuss our construction and
the difficulties involved with it by 
contrasting it to the Chern-Moser case. 
Before we describe the Chern-Moser 
normal form, 
let us comment shortly on why the differential geometric 
approach taken by Tanaka and Chern-Moser works in 
the case of hypersurfaces. The reason for this
is that actually locally, the geometric information induced 
by the (now scalar-valued!) Levi-form can be reduced to 
its signature and therefore stays, in a certain sense ``constant''. 
This makes it possible to study the structure using tools 
which are nowadays formalized under the umbrella of {\em parabolic geometry}--for further information, we 
refer the reader to the book of Cap and Slovak \cite{Cap:2009wu}. 
In particular, every Levi-nondegenerate hypersurface can be endowed 
with a structure bundle carrying a Cartan connection and 
an associated intrinsic curvature. However, in the 
case of Levi-nondegenerate manifolds of higher codimension, our 
basic second order invariant,  the 
vector-valued Levi form $\mathcal{L}_p$, has more invariants than 
just the simple integer-valued signature of a scalar-valued form, and 
its behaviour thus can (and in general will) change dramatically
with $p$. Of course, if it is nondegenerate at the given point $0$, 
 it stays so in 
 neighbourhood of it. There 
have thus been rather few circumstances 
in which the geometric approach has been successfully 
applied to Levi-nondegenerate manifolds of higher codimension, 
such as in the work of Schmalz, Ezhov, Cap, and others
 (see \cite{Schmalz:2006gt} and references therein). 

 In our paper, we take the different (dynamical systems inspired) approach taken by Chern-Moser,
 who introduced a {\em convergent normal form} for the problem. They 
 prescribe a space of {\em normal forms} 
 $\mathcal{N}_{CM} \subset \fps{ z, \bar z , s} $ such 
 that for each element of the infinitesimal automorphism algebra
 of the model hyperquadric $\imag w = \bar z^t J z$, one obtains
 a unique formal choice $(z,w)$ 
 of coordinates in $\CN = \Cn \times \C$ in which the defining 
 equation takes the form 
 \[ \imag w = \bar z^t J z + \Phi(z,\bar z, \real w),\]
 with $\Phi \in \mathcal{N}_{CM}$. It turns out (after
  the fact) that the coordinates are actually {\em holomorphic coordinates}, not only formal ones, which is the reason why 
  we say that the Chern-Moser normal form is {\em convergent}. 
  Let us shortly note that the dependence on the infinitesimal
  automorphism algebra is actually necessary; after all, some 
  of the hypersurfaces studied have a normal form 
  which still carries some symmetries (in particular, the 
  normal form of the model quadric will be the model quadric itself). 

The normal form space of Chern and Moser is described as follows. 
One needs to introduce the {\em trace operator} 
\[ T = \left(\dop{\bar z}\right)^t J \left(\dop{ z}\right) \] 
and the homogeneous parts in $z$ and $\bar z$ of a series 
$\Phi(z, \bar z, u) = \sum_{j,k} \Phi_{j,k} (z,\bar z, u)$,
 where 
  $\Phi_{j,k} (tz , s\bar z, u)= t^j s^k \Phi_{j,k} (z,\bar z, u)$; 
  $\Phi_{j,k}$ is said to be of {\em type} $(j,k)$.

We then say that $\Phi \in \mathcal{N}_{CM}$ if it satisfies the 
following (Chern-Moser) normal form conditions: 
\[ \begin{aligned}
\Phi_{j,0} &= \Phi_{0,j} = 0, &j\geq 0; \\
\Phi_{j,1} & = \Phi_{1,j} = 0, & j \geq 1; \\ 
T \Phi_{2,2} &= T^2 \Phi_{2,3} = T^3 \Phi_{3,3} = 0 . 
\end{aligned} \]
There are a number of aspects particular to the 
case $d=1$
which allow Chern and Moser to construct, based on 
these conditions (which arise 
rather naturally from a linearization of the 
problem with respect to the ordering by type), a convergent 
choice of coordinates. In particular, Chern and 
Moser are able to restate much of their problem in terms of 
ODEs, which comes from the fact that there is only one transverse 
variable when $d=1$; in particular, existence and regularity of solutions is guaranteed. In higher codimension, this changes dramatically, and we obtain systems of PDEs; neither do we 
a priori know that those are solvable nor do we know anything
about the regularity of their solutions (should they exist).
Our normal form has to take this into account.

Another aspect of the problem,  which really changes dramatically from the 
case $d=1$ to $d>1$, is the second line of the normal form 
conditions above: We cannot impose that $\Phi_{1,j} = \Phi_{j,1} = 0$
for $j\geq 1$, as those terms - it turns out - {\em actually
carry invariant information}. We shall however present 
a rather simple normal form, defined by equations which 
one can write down. 

We should note at this point that some parts of the problem
associated to a formal normal form have  already been studied
by Beloshapka \cite{Beloshapka:1990ib}. In there, a 
linearization of the problem is given, and a formal normal form
construction (with a completely arbitrary normal form space) is
discussed. However, for applications, a choice of a normal form 
space which actually gives rise to a convergent normal form
is of paramount importance, and only in very special 
circumstances (codimension $2$ in $\C^4$) there have 
been resolutions to this problem. 

The failure of a simple normalization of the 
terms of type $(1,j)$ and $(j,1)$ in the higher codimension case 
has more and subtle consequences which destroy much of the 
structure which allows one to succeed in the case $d=1$. We are 
able to overcome some of these problems 
by using a new technique from dynamical systems 
introduced by the second author \cite{Stolovitch:2014tk}. 
In that paper, one can already find an illustration of a kind 
of ``{higher codimension} Chern-Moser failure'' in a quite different but easier problem. 
It concerns normal forms of singularities of holomorphic functions. 
If the singularity is isolated, then usual proofs (Arnold-Tougeron) 
of the locally holomorphic conjugacy to a normal form reduces to the 
existence of holomorphic solutions of ODE's depending on a parameter 
(issued from ``la m\'ethode des chemins''). If the singularity is not 
isolated, there is no way to obtain such an ODE but the main 
result of  \cite{Stolovitch:2014tk}(Big denominator theorem) allows to solve the problem directly. 

In this paper we shall first discuss the convergent solution 
of a ``restricted'' (yet still infinite-dimensional) 
normalization problem:
Given a Levi-nondegenerate hyperquadric $\imag w = Q(z,\bar z)$, 
for perturbations of the form 
\[ \imag w = Q(z, \bar z) + \Phi_{\geq 3} (z,\bar z, \real w),\]
find a {\em formal} normal form.  
Our first main result can therefore be thought 
of as a concrete realization of Beloshapka's 
construction of an abstract normal form in this 
setting:
\begin{theorem}\label{thm:main1}
Fix a nondegenerate form of full rank $Q(z,\bar z)$ on $\Cn$ with 
values in $\Cd$, i.e. a map of the form $Q(z,\bar z) = (\bar z^t J_1 z, \dots ,\bar z^t J_d z)$ with the $J_k$ satisfying \eqref{e:nondegenerate}. Then there exists
a subspace $\hat{ \mathcal{N}}_f \subset \fps{z,\bar z , \real w}$ (explicitly given in
\eqref{e:defineNf} below) 
such that the following holds. 
Let $M$ be given near $0\in \CN$ by an equation of the form 
\[\imag w' = Q(z',\bar z',  ) + \tilde \Phi_{\geq 3} (z', \bar z' ,\real w'), \]
with $\tilde \Phi\in \fps{z,\bar z , \real w} $.
Then there exists a unique formal biholomorphic map of the form 
$H(z,w)= (z + f_{\geq 2} , w + g_{\geq 3})$ such that in the new (formal) coordinates
$(z,w)=H^{-1} (z',w')$ the manifold $M$ is given by an equation of the 
form 
\[ \imag w= Q(z,\bar z) + \Phi_{\geq 3} (z,\bar z, \real w) \]
with $\Phi_{\geq 3} \in \hat{\mathcal{N}}_f$. 
\end{theorem}
The solution of the {\em analytic} normal 
form problem, however, runs into all of the 
difficulties described above. However, there is a 
partial, ``weak'' normalization problem, 
described by a normal form space 
${\hat{ \mathcal{N}}}_f^w \supset {\hat{ \mathcal{N}}}_f$  (again 
defined below in \eqref{e:defineNf}), which 
in practice does not try to normalize the $(3,2)$ and
the $(2,3)$-terms and therefore treats the 
transversal $d$-manifold $z = f_0 (w)$ as 
a parameter. This fact is somewhat of independent 
interest, so we state it as a theorem: 
\begin{theorem}\label{thm:main2}
Fix a nondegenerate form of full rank $Q(z,\bar z)$ on $\Cn$ with 
values in $\Cd$, i.e. a map of the form $Q(z,\bar z) = (\bar z^t J_1 z, \dots ,\bar z^t J_d z)$ with the $J_k$ satisfying \eqref{e:nondegenerate}. Then for
the subspace $\mathcal{N}^w = \hat{\mathcal{N}}^w_f \cap \cps{z,\bar z , \real w}$ defined below in 
\eqref{e:defineNf}
the following holds. 
Let $M$ be given near $0\in \CN$ by an equation of the form 
\[\imag w' = Q(z',\bar z',  ) + \tilde \Phi_{\geq 3} (z', \bar z' ,\real w').\]
Then for any $f_0 \in (w) \cps{w}$  there exists a unique biholomorphic map of the form 
$H(z,w)= (z + f_0 + f_{\geq 2} , w + g_{\geq 3})$ with 
$f_{\geq 2} (0,w) = 0$ such that in the new coordinates
$(z,w)=H^{-1} (z',w')$ the manifold $M$ is given by an equation of the 
form 
\[ \imag w= Q(z,\bar z) + \Phi_{\geq 3} (z,\bar z, \real w) \]
with $\Phi_{\geq 3} \in \mathcal{N}^w$. 
\end{theorem}
Let us note that (as is apparent 
from the construction of the convergent solution) the corresponding formal problem also has a solution. 

Geometrically speaking, 
the convergent normal form given here provides 
for a unique convergent ``framing'' of the complex tangent spaces along and 
parametrization for any  germ of a real manifold $N\subset M$ transverse to $T^c_0 M$, i.e.
a map $\gamma\colon \Rd \to M$ parametrizing
$N$ and for each $t\in\Rd$, a basis 
of $T^c_{\gamma(t)} M$.

The {\em analytic} choice of such a transverse manifold satisfying the additional restrictions
to be in $\hat{\mathcal{N}}_f$
 is actually  quite more involved than the choice
 of a transverse curve in 
the case of a hypersurface, as the ``resonant terms'' already alluded to above provide for 
an intricate coupling of the PDEs which we 
will derive in their nonlinear terms. It is with 
that in mind that one has to put some 
additional constraint in order to provide for a 
complete normalization. We note, however, that 
we obtain a complete solution to the
formal normalization problem.

As already stated, in this generality we cannot guarantee convergence of 
the normal form. However, there are some 
{\em purely algebraic} conditions describing a subset of formal normal forms, for which 
the transformation to the normal form (and therefore also the normal form) can 
be shown to be convergent if the data is. 

\begin{theorem}\label{thm:main3}
Fix a nondegenerate form of full rank $Q(z,\bar z)$ on $\Cn$ with 
values in $\Cd$, i.e. a map of the form $Q(z,\bar z) = (\bar z^t J_1 z, \dots ,\bar z^t J_d z)$ with the $J_k$ satisfying \eqref{e:nondegenerate}. 
Let $M$ be given near $0\in \CN$ by an equation of the form 
\[\imag w' = Q(z',\bar z',  ) + \tilde \Phi_{\geq 3} (z', \bar z' ,\real w'), \]
with $\tilde \Phi\in \cps{z,\bar z , \real w} $.
Then any  formal biholomorphic map into the normal form 
from \cref{thm:main1} is convergent 
if the (formal) normal form
\[ \imag w= Q(z,\bar z) + \Phi_{\geq 3} (z,\bar z, \real w) \]
satisfies
\begin{equation}
\label{e:crucial1} \Phi_{1,1}' \Phi_{1,2} - \Phi_{1,2}'(Q + \Phi_{1,1})= 0.
\end{equation}
\end{theorem}

It is a natural question to ask how our normal form 
relates to the Chern-Moser normal form. In fact, our normalization procedure in 
\autoref{thm:main3} is a bit different from the Chern-Moser procedure. 
Let us emphasize that in the  hypersurface case ($d=1$) the normal form 
in \cref{thm:main1}, even though necessarily different from the Chern-Moser normal form, 
is automatically convergent. Indeed, in this case, 
\eqref{e:crucial1}  on the formal normal form is {\em automatically satisfied} since 
$\Phi_{1,1} = \Phi_{1,2} = 0$. 

The construction of our normal form is different than the Chern-Moser construction, 
since it is geared towards higher codimensional manifolds. However, we 
can adapt it in such a way that in codimension one, we obtain a completely 
new proof of the convergence of the Chern-Moser normal form, which relies completely 
on the inductive procedure used to construct it. We shall discuss this in detail
in \autoref{sec:chernmoser}.

\section{Framework}
We first gather some notational and technical preliminaries, which 
are going to be used in the sequel without further mentioning. 
\subsection{Initial  quadric}
Let $\tilde M$ be a germ of a real analytic manifold at the origin 
of $\Bbb C^{n+d}$ defined by an equation of the form
\begin{equation}\label{start-eq}
v'=Q(z',\bar z')+ \tilde \Phi_{\geq 3}(z',\bar z', u')
\end{equation}
where $w':=u'+iv'\in  \Bbb C^d$, $u' = \real w' \in \Rd$, $v' = \imag w' \in\Bbb R^d$ and $z'\in \Bbb C^n$. Here, $Q(z',\bar z')$ is a quadratic polynomial map with 
values in $\Bbb R^d$ and $\tilde \Phi_{\geq 3}(z',\bar z', u')$ an analytic map germ at $0$.
We endow  the variables $z',\bar z', w'$ with weights: $z'$ and 
$\bar z'$ get endowed with weights $p_1=p_2=1$ and $w'$ (and 
also $u$ and $v$) with $p_3=2$ respectively. Hence, the defining equation of the 
{\em model quadric} $\imag w = Q(z,\bar z)$ is quasihomogeneous (q-h) of quasi-degree (q-d) $2$.
We assume that the higher order deformation 
$\tilde \Phi_{\geq 3} (z,\bar z, u)$ has quasi-order (q-o) $\geq 3$, that is 
$$
 \tilde \Phi_{\geq 3}(z',\bar z', u')=\sum_{p\geq 3} \tilde\Phi_p(z',\bar z', u'),
$$
with $\tilde \Phi_p(z',\bar z', u')$ q-h of degree $p$. 
Hence, $\tilde M$ is a higher order perturbation of the quadric defined
by the homogeneous equation  $v' = Q(z', \bar z')$.
We assume that the quadratic polynomial $Q$ is a Hermitian form on $\Cn$, valued in $\Rd$, meaning it is  of the form
\[ Q(z,\bar z) = \begin{pmatrix} Q_1({z},{\bar z}) \\ \vdots \\ Q_d({z},{\bar z})
\end{pmatrix}, \]
where each $Q_k({z},{\bar z}) = {\bar z}^t J_k z$ is a Hermitian form 
on $\Bbb C^n$ defined by a Hermitian $n\times n$-matrix $J_k$. In particular, we observe that  $\overline{Q(a,\bar b)}=Q(b, \bar a)$, for any $a,b\in \Bbb C^n$.

We assume that $Q(z,\bar z)$ is {\em nondegenerate},  if $Q(v,e) = 0$ for all $v\in\Cn$ implies $e = 0$, or equivalently, 
\[ \bigcap_{k=1}^d \ker J_k = \{ 0 \}. \]
We also assume that the forms $J_k$ are \emph{linearly independent,} which translates to the fact that if $\sum_k \lambda_k J_k =0$ for 
scalars $\lambda_k$, then $\lambda_k = 0$, $k=1,\dots ,d$.  

In terms of the usual nondegeneracy conditions of CR geometry (see e.g. \cite{Baouendi:1999uy}) 
these conditions can be stated equivalently by requiring that  the
model quadric  $v = Q(z,\bar z)$ is $1$-nondegenerate and of finite type at the 
origin. 

\subsection{Complex defining equations} \label{ss:complexdef}
We will also have use for the complex defining equations for the real-analytic
(or formal) manifold $M$. If $M$ is given by 
\[ \imag w = \varphi(z, \bar z, \real w), \]
where $\varphi$ is at least quadratic, an application of the implicit function theorem (solving for $w$)
shows that one can give an equivalent equation 
\[ w = \theta (z, \bar z, \bar w).   \]
Such an equation comes from the defining equation of a real hypersurface if and only  if 
$\theta (z, \bar z , \bar \theta (\bar z, z, w)) = w$. 

We say that the coordinates $(z,w)$ are normal if $\varphi(z,0,u) = \varphi(0,\bar z, u) =0$, or 
equivalently, if $\theta (z, 0, \bar w) = \theta (0, \bar z, \bar w)$. The following 
fact is useful: 
\begin{lemma}
\label{lem:normal} Let $\varrho (z, \bar z, w, \bar w)$ be a defining function for a germ 
of a  real-analytic submanifold
$M\subset \C^n_z \times \C^d_w$. Then $(z,w)$ are normal coordinates for $M$ if and only 
if $\varrho (z, 0, w, w) = \varrho (0,\bar z, w, w) = 0$.
\end{lemma}
For a proof, we refer to \cite{Baouendi:1999uy}. 

\subsection{Fischer inner product}

Let $V$ be a finite dimensional vector space (over $\C{}$ or $\R{}$), endowed with
an inner product $\inp{\cdot}{\cdot}$. We denote by $u=(u_1,\dots, u_d)$ a 
(formal) variable, and write $\fpstwo{V}{u}$ for the space of formal power series
in $u$ with values in $V$. A typical element $f\in\fpstwo{V}{u}$ will be written 
as
\begin{equation}
  \label{e:fdecompose} f(u) = \sum_{\alpha\in\N^{d}} f_{\alpha} u^{\alpha}, \quad f_\alpha \in V.
\end{equation}
We define an extension of this inner product to $\fpstwo{V}{u} $ by 
\begin{equation}
  \label{e:fischer} \inp{f_\alpha u^\alpha}{g_\beta u^\beta} =
  \begin{cases}
   \alpha! \inp{f_\alpha}{g_{\alpha}} & \alpha = \beta \\
   0 & \alpha\neq \beta.
   \end{cases} 
\end{equation}
The inner product $\inp{f}{g}$ is not defined on all of $\fpstwo{V}{u}$, but is only  defined whenever at most finitely many of the 
products $f_{\alpha} g_{\alpha} $ are nonzero. In particular, $\inp{f}{g}$ is 
defined whenever $g\in F[u]$. This inner product is called the Fischer inner product \cite{fischer, Belitskii-formal}. If $T\colon \fps{F_1}{u} \to \fps{F_2}{u}$
is a linear map, we say that $T$ has a formal adjoint if there exists a 
map $T^* \colon \colon \fps{F_2}{u} \to \fps{F_1}{u} $ such that 
\[ \inp{Tf}{g}_2 = \inp{f}{T^* g}_1 \]
whenever both sides are defined. 

\begin{lemma}
\label{lem:existenceformaladjoint} A linear map $T$ as above has a formal 
adjoint if $T({F_1}[u]) \subset F_2 [u]$.
\end{lemma}

\begin{proof}
Let $T(f_\alpha u^\alpha) =: g^\alpha = \sum_\beta g^\alpha_\beta u^\beta$, and set $T^*( h_\beta u^{\beta}) = s^\beta (u) = \sum_\alpha s^{\beta}_{\alpha}{u^{\alpha}} $. 
We need that 
\[ \begin{aligned}
\inp{T(f_\alpha u^\alpha)}{h_\beta u^\beta}_2 &= \beta! \inp{g^\alpha_{\beta}}{h_\beta}_2 \\ 
& = \inp{f}{T^* ( h_\beta u^\beta)}_1 \\
& = \alpha! \inp{f_{\alpha}}{s^{\beta}_{\alpha}}_1,
 \end{aligned} \]
 which has to hold for all $\alpha,\beta$, and arbitrary $f_\alpha\in F_1, h_\beta \in F_2$. This condition determines $s_{\alpha}^{\beta} $ uniquely: 
 Fix $h_\beta$ and 
 consider the linear form 
 $ F_1\ni f_\alpha \mapsto \inp{T f_\alpha u^\alpha}{ h_\beta }. $ 
 Since $\inp{\cdot}{\cdot}_1$ is nondegenerate, there exists a uniquely 
 determined $s_\alpha^\beta \in F_1$ such that
  $\inp{g^\alpha_\beta}{h_\beta}_2 = \frac{\alpha!}{\beta!} \inp{f_\alpha}{s_\alpha^\beta}_1 $.

  We now only need to ensure that the series $T^* h$ is well-defined for 
  $h = \sum_\beta h_\beta u^\beta$. It would be given by 
  \[ T^*h = \sum_\alpha \left(\sum_\beta s^\beta_\alpha \right) u^\alpha, \]
 which  is a 
 well-defined expression under the condition that $T(f_\alpha u^\alpha)$ is a polynomial.
\end{proof}

We are now quickly going to review some of the facts 
and constructions which we are going to need. 

The map $D_\alpha\colon \fpstwo{F}{u} \to \fpstwo{F}{u}$, 
\[ D_\gamma f(u) = \vardop{f}{u}{\gamma} = \sum_\beta \binom{\alpha!}{\gamma!}{\gamma!} f_\alpha u^{\alpha-\gamma}  \]
has the formal adjoint 
\[ M_\gamma g(u) = u^\gamma g(u).\]
Indeed, 
\[ \inp{D_\gamma f_\alpha u^\alpha}{g_\beta u^\beta} = 
\begin{cases}
 \binom{\alpha}{\gamma} \gamma! (\alpha - \gamma)! \inp{f_\alpha}{g_{\alpha-\gamma}} = \inp{f_\alpha u^\alpha}{g_\beta u^{\beta+\gamma}} & \beta =\alpha - \gamma \\ 0 & \beta \neq \alpha - \gamma
 \end{cases}. \]

If $L\colon F_1 \to F_2$ is a linear operator, then the induced operator 
$T_L \colon \fpstwo{F_1}{u} \to \fpstwo{F_2}{u} $ defined by 
\[ T_L \left( \sum_\alpha f_\alpha u^\alpha \right) = \sum_\alpha Lf_\alpha u^\alpha \]
has the formal adjoint $T_L^* = T_{L^*}$, since
\[ \inp{T_L f_\alpha u^\alpha}{g_\beta u^\beta}_2 = 
\begin{cases}
 \alpha! \inp{Lf_{\alpha}}{g_\beta}_2 = \alpha!\inp{f_\alpha}{L^* g_\beta}_1 = 
 \inp{f_{\alpha} u^{\alpha}}{T_{L^*} g_\beta u^\beta} & \alpha = \beta \\ 0 &\text{else}.
 \end{cases} \]

Let $L_j \colon \fpstwo{F}{u} \to \fpstwo{F_j}{u} $ be linear operators, $j=1,\dots,n$, each of which possesses a 
formal adjoint $L_j^*$. Then the operator
\[ L= (L_1,\dots,L_n) \colon \fpstwo{F}{u} \to \fpstwo{ \oplus_j F_j }{u}, \]
where $\oplus_j F_j$ is considered as an orthogonal sum, 
has the formal adjoint $L^* = \sum_j L_j^* $.

More generally, it is often convenient to gather 
all derivatives together: 
consider the map $D_k \colon \fpstwo{F}{u} \to \fpstwo{\Sym^k F}{u}$, where
$\Sym^k F$ is the space of symmetric $k$-tensors on $\C^{d}$ (respectively 
$\R^{d}$) with values in $F$, defined 
by 
\[ D_k f(u) = \left( D_\alpha f (u) \right )_{\substack{\alpha\in\N^{d}\\|\alpha| = k}}\]
has the formal adjoint $D_k^* = M_k$ given by 
\[ M_k g(u) = \sum_{\substack{\gamma\in \N^{d} \\ |\gamma| = k }} 
g_\gamma(u) u^\gamma. \]
Here we realize the space $\Sym^k F$ as the space of homogeneous polynomials of 
degree $k$ in $d$ variables $(u_1,\dots,u_d)$, i.e. 
\[ \Sym^k F = \bigoplus_{j=1}^{\binom{k+d-1}{d-1} } F, \]
with the induced norm as an orthogonal sum (which is the usual induced norm on that space). 
\

If $L_1 \colon \fpstwo{F}{u}\to\fpstwo{F_1}{u} $ and $L_2 \colon \fpstwo{F_1}{u}\to\fpstwo{F_2}{u}$ are linear maps each of which possesses a formal adjoint, 
then $L = L_2 \circ L_1$ has the formal adjoint $L^* = L_1^* \circ L_2^*$.

It is often convenient to use 
the {\em normalized} Fischer product \cite{stolo-lombardi}, which is defined by 
\begin{equation}
  \label{e:fischernormalized} \inp{f_\alpha u^\alpha}{g_\beta u^\beta} =
  \begin{cases}
   \frac{\alpha!}{|\alpha|!} \inp{f_\alpha}{g_{\alpha}} & \alpha = \beta \\
   0 & \alpha\neq \beta.
   \end{cases} 
\end{equation}
While the adjoints with respect to 
the normalized and the standard Fischer inner product 
differ by 
constant factors for terms of the same homogeneity, the 
existence of adjoints and their kernels agree. Thus, it is 
not necessary to distinguish between the normalized and 
the standard Fischer product when looking at kernels of 
adjoints. The normalized version of the inner product is
far more suitable when dealing convergence issues and also better
for nonlinear problems \cite{stolo-lombardi}[proposition 3.6-3.7].

Our coefficient spaces $F_1$ and $F_2$ are often going to be spaces 
of polynomials (in $z$ and $\bar z$) of certain homogeneities, 
themselves equipped with the Fischer norm. 
Let $\mathcal{H}_{n,m}$ be the space of homogeneous polynomials of degree $m$ in $z\in\Cn$. We shall omit to write dependance on the dimension $n$ if the context permits.
Our definition of  the (normalized) Fischer inner product 
$\inp{\cdot}{\cdot}$, means that on monomials  
\begin{equation}
  \label{e:deffischer}  \inp{z^\alpha}{z^\beta} = \begin{cases}
    \frac{\alpha!}{|\alpha|!} & \alpha=\beta, \\
    0 & \alpha \neq \beta,
  \end{cases}
\end{equation}
and the inner product on $(\mathcal{H}_{n,m})^\ell$ is induced by declaring that the 
components are orthogonal with each other~:
if $f=(f^1,\dots,f^\ell)\in (\mathcal{H}_{n,m})^\ell$, 
then $\inp{f}{g} = \sum_{j=1}^\ell \inp{f^j}{g^j}$.

Let  $\mathcal{R}_{m,k}$ be the space of polynomials in $z$ and $\bar z$, valued in $\Cd$, which 
are homogeneous of degree $m$ (resp. $k$) in $z$ (resp. $\bar z$). 
Also this space will be equipped with the Fischer inner product $\inp{\cdot}{\cdot}_{d,k}$, where 
the components are declared to be orthogonal as well. That is, the inner product of a polynomial 
$P = (P_1,\dots,P_d)^t \in \mathcal{R}_{m,k}$
with a polynomial $Q = (Q_1,\dots,Q_d)^t \in \mathcal{R}_{m,k}$ is defined by 
$\inp{P}{Q} = \sum_\ell \inp{P_\ell}{Q_\ell}$, and  the latter inner products are given  on the basis monomials by
\begin{equation}
  \label{e:deffischer2}  \inp{z^{\alpha_1} \bar z^{\alpha_2}}{z^{\beta_1} \bar z^{\beta_2}} = \begin{cases}
    \frac{\alpha_1!\alpha_2!}{(|\alpha_1|+|\alpha_2|)!} & \alpha_1=\beta_1, \, \alpha_2=\beta_2 \\
    0 & \alpha_1 \neq \beta_1 \text{ or } \alpha_2\neq\beta_2.
  \end{cases}
\end{equation}


\subsection{The normalization conditions} 
\label{sec:the_normalization_conditions}
In this section, we shall discuss some of 
the operators which we are going to encounter
and discuss the normalization conditions 
used in \cref{thm:main1}, \cref{thm:main2}, and \cref{thm:main3}. The first normalization 
conditions on the $(p,0)$ and $(0,p)$ terms 
of a power series $\Phi (z,\bar z, u) \in \fps{z, \bar z ,u}$,
decomposed as 
\[ \Phi (z, \bar z, u) = \sum_{j,k = 0}^\infty \Phi_{j,k} (z,\bar z, u) ,  \]
is that 
\begin{equation}
\label{e:normalizep0} \Phi_{p,0} = \Phi_{0,p} = 0, \quad p\geq 0.
\end{equation}
With the potential to confuse the notions, we note 
that this corresponds to the requirement that $(z,w)$ are ``normal'' coordinates in the sense of Baouendi, Ebenfelt, and
Rothschild (see e.g. \cite{Baouendi:1999uy}) (it is also equivalent to  the requirement that $\Phi$ ``does not contain harmonic terms''). We write 
\begin{equation}
\label{e:defineN0} \mathcal{N}^0 := \{ \Phi \in \fps{z,\bar z, u} \colon \Phi(z,0,u) = \Phi (0,\bar z, u) = 0 \}.
\end{equation}

The 
first important operator, $\mathcal{K}$, is 
defined on formal power series in $z$ and $u$ (or $w$), and maps them to power series in $z, \bar z, u$, linear in $\bar z$,
by 
\[ \mathcal{K} \colon \fps{z,u}^d \to \faktor{\fps{z,\bar z , u}^d}{(\bar z^2)}, \quad  \mathcal{K} (\varphi(z,u)) = 
Q(\varphi(z,u),\bar z) = \begin{pmatrix}
  \bar z^t  J_1 (\varphi(z,u)) \\ \vdots \\ \bar z^t J_d (\varphi(z,u)) 
\end{pmatrix}.  \]
We can also consider $\bar{\mathcal K}$, defined by 
\[ \mathcal{\bar K} \colon \fps{\bar z,u}^d \to \faktor{\fps{z,\bar z , u}^d}{( z^2)}, \quad  \mathcal{K} (\varphi(\bar z,u)) = 
Q( z , \varphi(\bar z,u)) = \begin{pmatrix}
  (\varphi(\bar z,u))^t J_1 z \\ \vdots \\ (\varphi(\bar z,u))^t J_d z 
\end{pmatrix}.
\]
The important distinction for these operators
 to  the case $d=1$, is that for $d>1$, they
are not of full range. 
They are still injective, as we'll show later in \cref{lem:constantbound}. We will 
also construct a rather natural complementary 
space for their range, namely the kernels of 
\[ \mathcal{K}^* \colon \faktor{\fps{z,\bar z , u}^d}{(\bar z^2)} \to \fps{z,u}^d, \quad \mathcal{K}^* \begin{pmatrix}
  b_1 (z, \bar z, u) \\ \vdots \\ b_d (z, \bar z, u)
\end{pmatrix} = \sum_{j=1}^d \left(J_j \begin{pmatrix}
  \dop{\bar z_1}\big|_0 \\ \vdots \\ \dop{\bar z_n}\big|_0  
\end{pmatrix}\right)  b_j  \]
and of $(\bar{\mathcal K})^*$, respectively. These operators
are needed for the normalization of the $(p,1)$ and $(1,p)$
terms for $p > 1$ and constitute our first set of 
normalization conditions different from 
the Chern-Moser conditions: 
\begin{equation}
\label{e:normalizep1} \mathcal{K}^* \Phi_{p,1}  = \bar{\mathcal K}^* \Phi_{1,p} = 0, \quad p>1.
\end{equation}
We set the corresponding normal form space
\begin{equation}
\label{e:defineN1} \mathcal{N}^1_{\leq k} = \left\{ \Phi\in\fps{z,\bar z, u} \colon \mathcal{K}^* \Phi_{p,1} = \bar{\mathcal K}^* \Phi_{1,p} = 0 , 1< p \leq k\right\}. 
\end{equation}
For our other normalization conditions, in addition the operator $\mathcal K$, we 
shall need the operator $\Delta$, introduced by Beloshapka in  \cite{Beloshapka:1990ib}. 
It is defined for a power series map in $(z,\bar z, u)$ (valued in 
an arbitrary space) by 
\[ (\Delta \varphi)(z,\bar z, u)) = \sum_{j=1}^d \varphi_{u_j} (z,\bar z, u) Q_j (z,\bar z).\]
Its adjoint with respect to the Fischer inner product is going to 
play a prominent role: It is defined, again for an arbitrary power series map $\varphi$, by 
\[ \Delta^*  \varphi = \sum_{j=1}^d u_j Q_j \left(\dop{ z}, \dop{ \bar z} \right)\varphi. \]
The operator $\Delta^*$ is the equivalent to the trace operator which we are going to use.
The possible appearance of ``unremovable'' terms in $\Phi_{1,1}$ makes it a bit 
harder to formulate the corresponding trace conditions, as not only the obviously invariant $Q$ plays a role, but
rather all the invariant parts of  $\Phi_{j,j}$ for $j\leq 3$. Furthermore,
in the general setting, we do not have a ``polar decomposition'' for $\Phi_{1,1}$, making it 
hard to decide which terms to ``remove'' and which to ``keep'' when normalizing the 
diagonal tems. We opt for a balanced approach in our second set of normalization conditions, 
involving the diagonal terms $(1,1)$, $(2,2)$, and $(3,3)$: 
\begin{equation}
\label{e:normalizediagonal}
\begin{aligned}
-6 \Delta^* \Phi_{1,1} +(\Delta^*)^3 \Phi_{3,3} &= 0  \\
  \mathcal{K}^* ( \Phi_{1,1}  
- i  \Delta^* \Phi_{2,2} -   (\Delta^*)^2 \Phi_{3,3}) &= 0.
\end{aligned}
 \end{equation}
We define the set of power series $\Phi \in \fps{z,\bar z, u}$ satisfying 
these normalization conditions as $\mathcal{N}^d$ (``$d$'' stands for ``diagonal terms'').  Let us note that in the case $d=1$, these conditions are different from  the Chern-Moser conditions. 

The last set of normalization conditions deals with the $(2,3)$ and the $(3,2)$ terms; 
those possess terms which are not present in the Chern-Moser setting, but which simply disappear in the case $d=1$, reverting to the Chern-Moser conditions: 
\begin{equation}
\label{e:normalize23} \mathcal{K}^* (\Delta^*)^2 \left( \Phi_{2,3} + i \Delta \Phi_{1,2}  \right)= \mathcal{\bar K}^* (\Delta^*)^2 \left( \Phi_{3,2} - i \Delta \Phi_{2,1}  \right) = 0.
\end{equation}
The space of the power series which satisfy this condition will be denoted by 
\begin{equation}
\label{e:defineNo} \mathcal{N}^{\text{off}} = \left\{ \Phi\in\fps{z,\bar z, u} \colon \mathcal{K}^* (\Delta^*)^2 \left( \Phi_{2,3} + i \Delta \Phi_{1,2}  \right)= \mathcal{\bar K}^* (\Delta^*)^2 \left( \Phi_{3,2} - i \Delta \Phi_{2,1}  \right) = 0\right\}. 
\end{equation}
 This is the normal forms space of ``off-diagonal terms''. Let us note that in the 
 case $d = 1$, because 
 in our choice of normalization we have that  
 $\Phi_{1,1} \neq 0$ in general, even though our normalization 
 condition for the $(3,2)$ term reverts to  the same differential equation as 
 the differential equation for a chain, our full normal form will not necessarily produce chains. 
 We discuss this issue later in \autoref{sec:chernmoser}. 

We can now define the spaces 
$\hat{\mathcal{N}}_f \subset \hat{\mathcal{N}}_f^w $ of normal forms: 
\begin{equation}
\label{e:defineNf}
\begin{aligned}
\hat{\mathcal{N}}_f &:= \mathcal{N}^0 \cap \mathcal{N}^1_{\leq \infty} \cap \mathcal{N}^d \cap \mathcal{N}^{\text{off}}
\qquad \hat{\mathcal{N}}_f^w &:= \mathcal{N}^0 \cap \mathcal{N}^1_{\leq \infty} \cap \mathcal{N}^d
\end{aligned}
\end{equation}


\section{Transformation of a perturbation of the initial quadric}\label{sec:initialquadric}
We consider a formal holomorphic change of coordinates of the form
\begin{equation}\label{change}
z'=Cz+f_{\geq 2}(z,w)=:f(z,w),\quad w'=sw+g_{\geq 3}(z,w)=:g(z,w)
\end{equation}
where the invertible $n
\times n$ matrix $C$ and the invertible real $d
\times d$ matrix  $s$ satisfy
$$
Q(Cz,\bar C\bar z)=sQ(z,\bar z).
$$
In these new coordinates, equation \re{start-eq} reads
\begin{equation}\label{new-eq}
v=Q(z,\bar z)+\Phi_{\geq 3}(z,\bar z, u).
\end{equation}
This is the new equation of the manifold $M$ (in the 
coordinates $(z,w)$). We need to find the expression of $\Phi_{\geq 3}$. We  have
the following {\em conjugacy equation}:
\begin{eqnarray*}
sv+\imag(g_{\geq 3}(z,w)) &= & Q\left(Cz+f_{\geq 2}(z,w),\bar C\bar z+\bar f_{\geq 2}(\bar z,\bar w)\right) \\
&&+ \tilde \Phi_{\geq 3}\left(Cz+f_{\geq 2}(z,w),\bar C\bar z+\bar f_{\geq 2}(\bar z,\bar w), su +\real( g_{\geq 3}(z,w))\right).
\end{eqnarray*}
%
%
%
%
Let us set as notation $f:=f(z,u+iv)$ and $\bar f:=\bar f(\bar z, u-iv)$ with $v:= Q(z,\bar z)+ \Phi_{\geq 3}(z,\bar z, u)$. We shall write $Q$ for $Q(z,\bar z)$.
The conjugacy equation reads
\begin{equation}\label{conj}
\frac{1}{2i}\left(g-\bar g\right)  =  Q\left(f,\bar f\right)+\tilde \Phi_{\geq 3}\left(f,\bar f, \frac{g+\bar g}{2}\right).
\end{equation}

As above, we set $f_{\geq2}:= f_{\geq2}(z, u+iv)$ and $\bar f_{\geq2}:= \bar f_{\geq2}(\bar z, u-iv)$. We have 
\begin{equation}
\begin{aligned}
 \frac{1}{2i}(s(u+iv)-s(u-iv)) &= sQ(z,\bar z)+ s\Phi_{\geq 3}(z,\bar z, v)\\
Q\left(f,\bar f\right) & =  Q\left(Cz+ f_{\geq 2},\bar C\bar z +\bar f_{\geq 2}\right)\\
&= Q\left(Cz ,\bar f_{\geq 2}\right)+ Q\left(f_{\geq 2},\bar C\bar z \right) + Q\left(Cz,\bar C\bar z \right) +Q\left(f_{\geq 2},\bar f_{\geq 2}\right)\\
\tilde \Phi_{\geq 3}\left(f,\bar f, \frac{1}{2}\left[g+\bar g\right]\right)& = \tilde \Phi_{\geq 3}\left(Cz,\bar C\bar z, su\right)\\
&\quad + \left(\tilde \Phi_{\geq 3}\left(f,\bar f, \frac{1}{2}\left[g+\bar g\right]\right)- \tilde \Phi_{\geq 3}\left(Cz,\bar C\bar z, su\right)\right) 
 \end{aligned} \end{equation}


Therefore, we can rewrite \eqref{conj} in the following way:  
\begin{equation}
\begin{aligned}
  \frac{1}{2i}  [  g_{\geq 3}(z,u+iQ)&-\bar g_{\geq 3}(\bar z,u-iQ)] 
-\left(Q\left(Cz,\bar f_{\geq 2}(\bar z,u-iQ)\right)+ Q\left(f_{\geq 2}(z,u+iQ),\bar C\bar z \right)\right) \\ &=  Q\left(f_{\geq 2},\bar f_{\geq 2}\right) \\ 
& \quad +\tilde \Phi_{\geq 3}\left(Cz,\bar C\bar z, su\right)-s\Phi_{\geq 3}(z,\bar z, u)\nonumber\\
& \quad + \left(\tilde \Phi_{\geq 3}\left(f,\bar f, \frac{1}{2}(g+\bar g)\right)- \tilde \Phi_{\geq 3}\left(Cz,\bar C\bar z, su\right)\right) \nonumber\\
& \quad +\frac{1}{2i}\left ( g_{\geq 3}(z,u+iQ)-g_{\geq 3}\right)\nonumber -\frac{1}{2i}\left(\bar g_{\geq 3}(\bar z,u-iQ)-\bar g_{\geq 3}\right)\nonumber\\
& \quad +\left( Q\left(Cz,\bar f_{\geq 2}\right)-Q\left(Cz,\bar f_{\geq 2}(z,u-iQ)\right)\right)\nonumber\\
& \quad + \left( Q\left(f_{\geq 2},\bar C\bar z \right)- Q\left(f_{\geq 2}(z,u+iQ),\bar C\bar z \right)\right)\label{big-equiv}
\end{aligned}\end{equation}

Let us set $C=\id$ and $s=1$. We shall write this equation as
\begin{equation}\label{conj-abstract}
{\mathcal L}(f_{\geq 2},g_{\geq 3}) = {\mathcal T}(z,\bar z, u;f_{\geq 2},g_{\geq 3},\Phi)-\Phi
\end{equation}
where ${\mathcal L}(f_{\geq 2},g_{\geq 3})$ (resp. ${\mathcal T}(z,\bar z, u;,f_{\geq 2},g_{\geq 3},\Phi)$) denotes the linear (resp.  nonlinear) operator defined on the left (resp. right) hand side of \re{big-equiv}.
The linear operator $\mathcal{L}$ maps the space the space of quasihomogeneous holomorphic vector fields $QH_{k-2}$ of quasi degree $k-2\geq 1$, that is, of expressions of the form
\[ f_{k-1} (z,w) \dop{z} + g_k (z,w) \dop w = f_{k-1} (z,w) \cdot \begin{pmatrix}
  \dop{z_1} \\ \vdots  \\ \dop{z_n}
\end{pmatrix} + g_k (z,w) \cdot \begin{pmatrix}
  \dop{w_1} \\ \vdots  \\ \dop{w_d}
\end{pmatrix}, \]
where $f_{k-1}$ and $g_k$ are quasi-homogeneous polynomials taking values in $\Cn$ and  $\Cd$, respectively to the space of quasi-homogeneous polynomials of degree $k\geq 3$ with values in $\Cd$.  We shall denote  the restriction of ${\mathcal L}$ to $QH_{k-2}$ by ${\mathcal L}_k$.

By expanding into quasihomogeneous component,  equation \re{conj-abstract} reads
\begin{equation}
{\mathcal L}(f_{k-1},g_{k}) = \{{\mathcal T}(z,\bar z, u;f_{\geq 2},g_{\geq 3},\Phi)\}_{k}-\Phi_{k} =  \{{\mathcal T}(z,\bar z, u; f_{\geq 2}^{<k-1},g_{\geq 3}^{<k}),\Phi_{<k}\}_{k}-\Phi_k.\label{conjugacy-k}
\end{equation}
Here,  $\{{\mathcal T}(z,\bar z, u; f_{\geq 2},g_{\geq 3}),\Phi\}_{k}$ (resp. $f_{\geq 2}^{<k-1}$) denotes the quasi-homogeneous term of degree $k$ (resp. $<k-1$) of the Taylor expansion of ${\mathcal T}(z,\bar z, u; f_{\geq 2},g_{\geq 3},\Phi)$ (resp. $f_{\geq 2}$) at the origin.

It is well-known (see e.g. \cite{Baouendi:1998uo}) 
that the operator $\mathcal{L}$, considered 
as  an operator on the space of (formal) holomorphic vector 
fields, under our assumptions of 
linear independence and nondegeneracy of the form $Q$, has a 
finite-dimensional (as a real vector space) kernel, which coincides 
with the space of infinitesimal CR automorphisms of the model quadric 
$\imag w = Q(z,\bar z)$ fixing the origin. It follows that, for any $k\geq 3$, any complementary 
subspace $\mathcal N_k$ to the image of $\mathcal{L}_k $ gives rise to 
a {\em formal normal form} of degree $k$. By induction on $k$, we prove that there exists a $(f_{k-1},g_k)$ and a $\Phi_{k}\in {\mathcal N}_k$ such that equation \re{conjugacy-k} is solved. 
 A a consequence, up to elements of the space 
of infinitesimal automorphisms of the model quadric, there exists a unique 
formal holomorphic change of coordinates such that the ``new'' defining function
lies in the space of normal form $\mathcal{N}:=\bigoplus_{k\geq 3}{\mathcal N}_k$. 

In order to find a way to choose $\mathcal{N}$ with the additional property 
that for analytic defining functions, the change of coordinates is 
also analytic, we shall pursue a path which tries to rewrite the 
important components of $\mathcal{L}$ as partial differential operators. 

From now on, we write $\{h\}_{p,q}$ for the term  in the Taylor expansion of $h$ which is homogeneous of degree $p$ in  $z$ and of degree $q$ in $\bar z$. 
For a map $h= h(z,\bar z, u)$,  
we have $\{h\}_{p,q}=h_{p,q}(u)$ for some map $h_{p,q}(u)$ taking
values in the space of polynomials homogeneous of degree $p$ in  $z$ and of degree $q$ in $\bar z$ (with values in the same space as $h$), which is  analytic in a fixed domain of $u$ independent of $p$ and $q$ (provided that $h$ is analytic). 
We also will from now on write $f_k(z,u)$ for  the homogeneous polynomial of degree $k$ (in $z$) in the Taylor expansion of $f$. Even though this 
conflicts with our previous use of the subscript, no problems shall arise 
from the dual use. 

In what follows our notation can be considered as an abuse of notation: 
in an expression such as $D^k_ug(z,u)(Q+\Phi)^k$, we write as if $Q+\Phi$ 
was a scalar. This is harmless since we are only interested in a lower bound of 
the vanishing order of some fix set of monomials in $z,\bar z$. However,
if one decides to consider $D^k_u g$ as a symmetric multilinear 
form and considers powers as appropriate ``filling'' of these 
forms by arguments, one can also consider the equations as 
actual equalities. 

 We have
\begin{equation}\label{g}
g_{\geq 3}(z,u+iQ)-g_{\geq 3}(z,u+iQ+i\Phi)=\sum_{k\geq 1}\frac{i^k}{k!}D^k_ug_{\geq 3}(z,u)\left(Q^k-(Q+\Phi)^k \right),
\end{equation}
and
$$
Q\left(f_{\geq 2}-f_{\geq 2}(z,u+iQ),\bar C\bar z \right)= Q\left(\sum_{k\geq 1}\frac{i^k}{k!}D^k_uf_{\geq 2}(z,u)\left(Q^k-(Q+\Phi)^k \right),\bar C\bar z \right),
$$
and therefore
$$
\left\{D^k_ug(z,u)\left(Q^k-(Q+\Phi)^k \right)\right\}_{p,q} = \sum_{l=0}^pD^k_ug_l(z,u)\left\{Q^k-(Q+\Phi)^k\right\}_{p-l,q} 
$$
and 
\begin{eqnarray}
\left\{Q\left(f_{\geq 2}-f_{\geq 2}(z,u+iQ),\bar C\bar z \right)\right\}_{p,q} &= &Q\left(\left\{f_{\geq 2}-f_{\geq 2}(z,u+iQ)\right\}_{p,q-1},\bar C\bar z \right)\label{Qf}\\
&=& \sum_{l=0}^p\sum_{k\geq 1}\frac{i^k}{k!}Q\left(D^k_uf_l(z,u)\left\{Q^k-(Q+\Phi)^k\right\}_{p-l,q-1} ,\bar C\bar z\right).\nonumber
\end{eqnarray}

\section{Equations for the \texorpdfstring{$(p,q)$}{(p,q)}-term of the conjugacy equation}\label{sec:equationsterms}

For any non negative integers $p,q$, let us set
$$
T_{p,q}:=\left\{\tilde \Phi_{\geq 3}\left(f,\bar f, \frac{1}{2}(g+\bar g)\right)- \tilde \Phi_{\geq 3}\left(Cz,\bar C\bar z, su\right)\right\}_{p,q}.
$$
\subsection{\texorpdfstring{$(p,0)$}{(p,0)}-terms}

According to \re{5p0}, \re{6p1},\re{Qffp0} 
, the $(p,0)$-term of the conjugacy equation \re{conj}, for $p\geq 2$, is
\begin{equation}
\frac{1}{2i}g_{p}=Q(f_{p},\bar f_0)+T_{p,0}+\tilde \Phi_{p,0}\left(Cz,\bar C\bar z, su\right)-s\Phi_{p,0}(z,\bar z, u)=:F_{p,0}.
\label{p0}
\end{equation}
For $p=1$, the linear map $\mathcal L$ gives a new term $-Q(Cz,\bar f_0)$ to the previous one. Hence, we have
\begin{equation}
\frac{1}{2i}g_1-Q(Cz,\bar f_0)=Q(f_1,\bar f_0)+T_{1,0}+\tilde \Phi_{1,0}\left(Cz,\bar C\bar z, su\right)-s\Phi_{1,0}(z,\bar z, u)=:F_{1,0}.
\label{10}
\end{equation}
For $p=0$, we have
\begin{equation}\label{00}
\imag (g_0)=Q(f_0,\bar f_0)+T_{0,0}+\tilde \Phi_{0,0}\left(Cz,\bar C\bar z, su\right)-s\Phi_{0,0}(z,\bar z, u)=:F_{0,0}
\end{equation}

\subsection{\texorpdfstring{$ (p,1) $}{ (p,1) }-terms}

According to \re{5p1}, \re{6p1},\re{Qffp1} 
, the $(p,1)$-term of the conjugacy equation \re{conj}, for $p\geq 3$, is 
\begin{eqnarray}
\frac{1}{2}D_ug_{p-1}Q-Q(f_p,\bar C\bar z)&=&\imag \left(iD_ug_{p-2}(u)\Phi_{2,1}+iD_ug_{p-1}(u)\Phi_{1,1}\right)+Q(f_p,\bar f_1)\nonumber\\
&&+iQ(Df_{p-1}(Q+\Phi_{1,1}), \bar f_0)-iQ(f_{p-1}, D_u\bar f_0(Q+\Phi_{1,1}))\nonumber\\
&&+\tilde \Phi_{p,1}(Cz,su)-s\Phi_{p,1}(z,u)+T_{p,1}=:F_{p,1}.
\label{p1}
\end{eqnarray}
For $p=2$, we get the same expression on the right hand side, but the linear part gains the term $iQ(Cz, D_u\bar f_0Q)$. Hence, we have
\begin{equation}
\frac{1}{2}D_ug_{1}Q-Q(f_2,\bar C\bar z)+iQ(Cz, D_u\bar f_0Q)=F_{2,1}.\label{21}
\end{equation}
For $p=1$, we have
\begin{equation}
D_u \real (g_0(u)) \cdot  Q -Q(Cz,\bar f_1(\bar z,u))-Q(f_1(z,u),\bar C\bar z)= F_{1,1}\label{11}
\end{equation}
\subsection{\texorpdfstring{$ (3,2) $}{ (3,2) }}

For the $(3,2)$-terms, we obtain
\begin{equation}\begin{aligned}
-\frac{1}{4i}D^2_ug_1(z,u)Q^2&+\frac{1}{2}Q(Cz,D_u^2\bar f_0(u)Q^2)-iQ(D_uf_2(z,u)Q, \bar C\bar z)= \re{Qff32}+\frac{1}{2i}\re{532}+\re{632}\\
&+\tilde \Phi_{3,2}(Cz,\bar C\bar z,su)-s\Phi_{3,2}(z,\bar z, u)
-\frac{1}{2i}\overline{\re{532}}+\overline{\re{632}}
+ \re{bigsum}_{3,2}.
\end{aligned}\label{32}
\end{equation}
where $\re{bigsum}_{3,2}$ denotes the $(3,2)$-component of $\re{bigsum}$, $\overline{\re{532}}$ (resp. $\overline{\re{632}}$) denotes the $(3,2)$-component 
of $\left(\bar g_{\geq 3}(\bar z,u-iQ)-\bar g_{\geq 3}\right)$
 (resp. $\left( Q\left(Cz,\bar f_{\geq 2}\right)-Q\left(Cz,\bar f_{\geq 2}(z,u-iQ)\right)\right)$).

\subsection{ \texorpdfstring{$ (2,2) $}{ (2,2) }-terms}
For the $(2,2)$ term, we have
\begin{equation}\begin{aligned}
-\frac{1}{2}D^2_u \imag(g_0) \cdot Q^2+& iQ(Cz, D_u\bar f_1(\bar z,u)\cdot Q)
-iQ(D_uf_1(z,u)\cdot Q,\bar C\bar z)= \re{Qff22}+\frac{1}{2i}\re{522}+\re{622}\\\
&+\tilde \Phi_{2,2}(Cz,\bar C\bar z,su)-s\Phi_{2,2}(z,\bar z, u) - \frac{1}{2i}\overline{\re{522}}+\overline{\re{622}}
+ \re{bigsum}_{2,2}=: F_{2,2}.
\end{aligned}
\label{22}
\end{equation}

\subsection{\texorpdfstring{$ (3,3) $}{ (3,3) }-terms}

For the $(3,3)$ term, we have
\begin{equation}\label{33}\begin{aligned}
-\frac{1}{6}D^3_u \real (g_0)\cdot Q^3 &+Q(Cz, D_u^2\bar f_1(\bar z,u)\cdot Q^2)+Q(D_u^2f_1(z,u)\cdot Q^2,\bar C\bar z)= \re{Qff33}+\frac{1}{2i}\re{533}+\re{633}\\
&+\tilde \Phi_{3,3}(Cz,\bar C\bar z,su)-s\Phi_{3,3}(z,\bar z, u) - \frac{1}{2i}\overline{\re{533}}+\overline{\re{633}}+ \re{bigsum}_{3,3}=: F_{3,3}.
\end{aligned}
\end{equation}

\section{A full formal normal form: Proof of \cref{thm:main1}} 
\label{sec:a_full_formal_normal_form_proof_of_cref}
We recall that we have used above the following notation for the grading of the transformation~: we consider transformations of the form 
\[z^*=z+\sum_{k\geq 0}f_{k}, w^*=w+\sum_{k\geq 0}g_{k}\] 
where $f_k(z,w)$ and $g_{k}(z,w)$ are homogeneous of degree $k$ in $z$; $f_k$ and 
 $g_k$ can also be considered as power series maps in $w$ valued in the 
 space of holomorphic polynomials in $z$ of degree $k$ taking values in $\Cn$ and $\Cd$,
 respectively. We then collect from the equations computed in \cref{sec:equationsterms}:
Using \re{00}, \re{p0} and \re{p1}, we have 
\begin{eqnarray*}
\imag(g_0) &=& F_{0,0}\\
\frac{1}{2i}g_{p}&=&F_{p,0}\\
\frac{1}{2}D_ug_{p}Q-Q(f_{p+1},\bar z)&=&F_{p+1,1}
\end{eqnarray*}
Using \re{21} and \re{32}, we have
\begin{eqnarray*}
\frac{1}{2}D_ug_{1}Q-Q(f_2,\bar z)+iQ(z, D_u\bar f_0Q)&=&F_{2,1}\\
-\frac{1}{4i}D^2_ug_1(z,u)Q^2+\frac{1}{2}Q(z,D^2\bar f_0(u)Q^2)-iQ(D_uf_2(z,u)Q, \bar z)&=&F_{3,2}
\end{eqnarray*}
Using \re{11},\re{22} and \re{33}, we have $\imag (g_0)=F_{0,0}$
\begin{eqnarray*}
D_u \real (g_0(u)) \cdot  Q -Q(z,\bar f_1(\bar z,u))-Q(f_1(z,u),\bar z)&= &F_{1,1}\\
-\frac{1}{2}D^2_u \imag{(g_0)} \cdot Q^2 +  iQ(z, D_u\bar f_1(\bar z,u)\cdot Q)-iQ(D_uf_1(z,u)\cdot Q,\bar z)&=& F_{2,2}\\
-\frac{1}{6}D^3_u \real (g_0)\cdot Q^3 +Q(z, D_u^2\bar f_1(\bar z,u)\cdot Q^2)+Q(D_u^2f_1(z,u)\cdot Q^2,\bar z)&=&F_{3,3}
\end{eqnarray*}
In order to obtain an operator $\mathcal{L}$ acting 
on the space of maps, and taking values in the space of formal power series in $\fps{z,\bar z,u}^d$ endowed with Hermitian product \ref{e:deffischer2}, we simplify a bit the left hand sides, 
express the linear occurence of the terms $\Phi_{p,q}$ of the ``new'' manifold, and 
change the right hand side accordingly: 
\begin{equation}
 \label{e:formalsys} \begin{aligned}
 \imag g_0 &= \Phi_{0,0} + \tilde F_{0,0} \\
 \frac{1}{2i} g_p &= \Phi_{p,0} + \tilde F_{p,0} \\
 - Q(f_{p+1}, \bar z) &= \Phi_{p+1, 1} + \tilde F_{p+1,1} \\
 - Q (f_2 , \bar z) + i Q(z,D_u \bar f_0 Q) &= \Phi_{2,1} + \tilde F_{2,1}\\
 \frac{1}{2}Q(z,D^2\bar f_0(u)Q^2)-iQ(D_uf_2(z,u)Q, \bar z)&=\Phi_{3,2} + \tilde F_{3,2}\\
 D_u \real (g_0(u)) \cdot  Q -Q(z,\bar f_1(\bar z,u))-Q(f_1(z,u),\bar z)&= \Phi_{1,1} + \tilde F_{1,1}\\
 iQ(z, D_u\bar f_1(\bar z,u)\cdot Q)-iQ(D_uf_1(z,u)\cdot Q,\bar z)&= \Phi_{2,2} + \tilde F_{2,2}\\
-\frac{1}{6}D^3_u \real (g_0)\cdot Q^3 +Q(z, D_u^2\bar f_1(\bar z,u)\cdot Q^2)+Q(D_u^2f_1(z,u)\cdot Q^2,\bar z)&= \Phi_{3,3} + \tilde F_{3,3}
 \end{aligned}
 \end{equation}
 At this point, the {\em existence of some formal normal form} 
 follows by studying the injectivity of the linear operators
 appearing on the left hand side of \eqref{e:formalsys} (as 
 already explained in Beloshapka \cite{Beloshapka:1990ib}). 
 We now explain how we 
 can reach the normalization conditions from \cref{sec:the_normalization_conditions}. 

 For the terms $\Phi_{p,0}$ (for $p\geq 0$) this is simply done by 
 applying the conditions \eqref{e:normalizep0} to \eqref{e:formalsys}
 and substituting the resulting expressions for 
 $\imag g_0$ and $g_p$ into the remaining equations. 

 In order to obtain the normalization 
 conditions for the terms 
 $\Phi_{p,1}$, we apply the operator
  $\mathcal{K}^*$ to lines $3$ and $4$ of \eqref{e:formalsys}, 
  yielding after application of the normalization 
  conditions \eqref{e:normalizep1}
  a system of implicit equations for $f_p$ for $p\geq 2$.
 If we substitute the solution of this problem back into the remaining equations, we obtain
 (now already using the operator notation)
 \begin{equation}
 \label{e:formalsys2} \begin{aligned}
 -\frac{1}{2} \bar{\mathcal{K}} \Delta^2 f_0&=\Phi_{3,2} - i\Delta \Phi_{2,1}  
  + {\hat F_{3,2}}\\
 \Delta \real (g_0) - \bar{\mathcal{K}} \bar f_1 - {\mathcal{K}}f_1&= \Phi_{1,1} + \hat F_{1,1}\\
 i\bar{\mathcal{K}}\Delta f_1-i {\mathcal{K}} \Delta f_1&= \Phi_{2,2} + \hat F_{2,2}\\
-\frac{1}{6} \Delta^3 \real (g_0) + \bar{\mathcal{K}} \Delta^2 \bar f_1+ {\mathcal{K}} \Delta^2 f_1&= \Phi_{3,3} + \hat F_{3,3},
 \end{aligned}
 \end{equation} 
We can then define the space of normal forms to be the kernel of the adjoint of the 
operator $\mathcal{L} \colon \fps{u}^n \times \fpstwo{R}{u}^d \times \fps{u}^{n^2} \to \mathcal{R}^d_{3,2} \oplus \mathcal{R}^d_{1,1} \oplus \mathcal{R}^d_{2,2} \oplus \mathcal{R}^d_{3,3}  $
\[\mathcal{L} (f_0, \real g_0 , f_1) = \begin{pmatrix}
  -\frac{1}{2} \bar{\mathcal{K}} \Delta^2 f_0\\
 \Delta \real (g_0) - \bar{\mathcal{K}} \bar f_1 - {\mathcal{K}}f_1\\
 i\bar{\mathcal{K}}\Delta \bar f_1-i {\mathcal{K}} \Delta f_1\\
-\frac{1}{6} \Delta^3 \real (g_0) + \bar{\mathcal{K}} \Delta^2 \bar f_1+ {\mathcal{K}} \Delta^2 f_1
\end{pmatrix}\] with respect to the Hermitian products on 
these spaces. The solution can be found by constructing the 
homogeneous terms in $u$ (!) of $f_0$, $\psi$, $f_1$ inductively,
since the right hand sides only contains terms of lower order 
homogeneity (and thus, found in a preceding step). However, the 
$f_1$ enters the nonlinear terms in such a way as to render 
the system \eqref{e:formalsys2} {\em singular} when 
one attempts to interpret it as (a system of complete partial) differential equations, because the equation for the $(3,2)$-term
contains in the $\tilde F_{3,2}$ an $f_1''$, thereby linking 
$\bar f_0'$ with $f_1''$; therefore, the appearance 
of $f_0'''$ in the term $\tilde F_{3,3}$ acts as if
it contained an $f_1'''$, which exceeds the order of 
derivative $f_1''$ appearing in the linear part. 

However, in the formal sense, a solution to this equation
exists and is unique 
modulo $\ker \mathcal{L}$, which we know to be a finite dimensional space, and in particular 
unique if we require  $(f_0, \real g_0 , f_1) \in \imag \mathcal{L}^*$. This gives us exactly
exactly our normal form space, and thus gives \cref{thm:main1}.

\section{Analytic solution to the weak conjugacy problem: Proof of \cref{thm:main2}}
\subsection{Step 1: Preparation}
In this section, we shall first  find a change of coordinates 
of the form $z'=f_0 (w) + z $ and $w'=w +  i G(z,w)$, where $G (0,w)= \bar G (0,w)$,  
in order  to ensure the 
 normalization conditions  $\Phi_{p,0}=\Phi_{0,p}=0$  for all non negative integers $p$.
 This condition is equivalent to the fact that the coordinates $(z,w)$ are {\em normal} in 
 the sense of \cref{ss:complexdef}. In particular, if we consider a complex defining 
 equation $\tilde\theta$ for our perturbed quadric $\imag w' = Q (z', \bar z') + \tilde \Phi (z', \bar z', \real w')$, then we see by 
 \cref{lem:normal} that 
 $(z,w)$ are normal coordinates if and only if
\begin{equation}
\label{e:phi00prepa}
 w + i G(z,w) = \tilde \theta (z + f_0 (w), \bar f_0 (w), w - i G(0,w) ),  
\end{equation}
or eqivalently if and only if 
\begin{equation}\label{e:phi00prep}
\frac{1}{2} \left( G(z,w) + \bar G (0,w)  \right) = \tilde \varphi \left(z + f_0 (w) , \bar f_0 (w), w + \frac{i}{2} (G(z,w) - \bar G(0,w))  \right)
\end{equation}
We can thus first obtain $G(0,w)$ from  the equation derived from \eqref{e:phi00prep} by putting $z=0$: 
\[  G(0,w) = \tilde \varphi (f_0 (w), \bar f_0 (w), w ) \] 
and then define $G(z,w)$ by \eqref{e:phi00prepa},
obtaining 
\[ G(z,w)= \frac{1}{i} \left( \tilde\theta(z+ f_0(w), \bar f_0(w), w - i \tilde\varphi(f_0 (w), \bar f_0 (w), w )) - w \right). \]
 Summing up: we can therefore replace the given defining funtion by this new one, and assume 
from now on that $f_0 =0$ and that the coordinates are already normal. This change of coordinates is rather standard and can be 
found in e.g. \cite{Baouendi:1999uy}.

\subsection{Step 2: Normalization of \texorpdfstring{$ (1,1)$, $(2,2)$, $(3,3)$, and  $(2,1)$}{ (1,1), (2,2), (3,3), and  (2,1)}-terms}

In this section we shall normalize further the equations of the manifold. Namely, we shall proceed a change of coordinates such that, not only, the manifold is prepared as in the previous section, but also its $(1,1)$, $(2,1)$, $(2,2)$, and $(3,3)$ terms belong to a subspace of normal forms. 
We will now (after having prepared with the given map $f_0$) 
only consider a change of coordinates 
of the form $z'=z+f(z,w)=z+f_1+f_2$ and $w'=w+ g(z,w)=w+g_0$ which satisfies
 $f(0,w)=0$, $g(0)=0$ and $Df(0)=0$, $Dg(0)=0$. We assume 
 that 
$\Phi_{p,0}=\Phi_{0,p}=0$, $0\leq p$, i.e. that 
$g$ has been chosen according to the solution 
of the implicit function theorem in the preceding subsection; with 
the preparation above, i.e. $\tilde \Phi_{p,0} = \tilde \Phi_{0,p} = 0$, and 
the restriction on $f$ this amounts to $\imag g_0 = 0$.
 Using the left hand side of equations \re{11}, \re{22},\re{33},\re{21} and \re{32} together with $f_0=0$, let us set
\begin{eqnarray}
L_{1,1}(f_1, g_0) &:=& D_u\real(g_0(u))\cdot Q -Q(z,\bar f_1(\bar z,u))-Q(f_1(z,u),\bar z)\label{l11}\\
L_{2,2}(f_1, g_0) &:=&\frac{-1}{2}D^2_u\imag(g_0)\cdot Q^2+iQ(z, D_u\bar f_1(\bar z,u)\cdot Q)-iQ(D_uf_1(z,u)\cdot Q,\bar z)\label{l22}\\
L_{3,3}(f_1, g_0) &:=&\frac{-1}{6}D^3_u\real(g_0)\cdot Q^3+Q(z, D_u^2\bar f_1(\bar z,u)\cdot Q^2)+Q(D_u^2f_1(z,u)\cdot Q^2,\bar z)\label{l33}\\
L_{2,1}(f_2) & = &
-Q(f_2,\bar z)\label{l21}\\
L_{3,1}(f_3) & = &
-Q(f_3,\bar z)\label{l31}\\
\end{eqnarray}

Therefore, equations \re{11},\re{22} and \re{33} read~:
\begin{eqnarray}
L_{1,1}(f_1, g_0)&=&\real (D_ug_{0}(u))\Phi_{1,1}+Q(f_1,\bar f_1)\nonumber\\
&&+\tilde \Phi_{1,1}(z,\bar z,u)-\Phi_{1,1}(z,\bar z, u)\nonumber\\
&&+D_z\tilde \Phi_{1,1}(z,\bar z,u)f_1(z,u)+D_{\bar z}\tilde \Phi_{1,1}(z,\bar z,u)\overline{ f_1(z,u)}\label{matrix11}\\
L_{2,2}(f_1, g_0)&=& iQ(D_uf_1(Q+\Phi_{1,1}),\bar f_1)-iQ(f_1, D_u\bar f_1(Q+\Phi_{1,1}))\nonumber\\
&&+2\real(Q(iD_uf_1(u)\Phi_{1,1},\bar z))+ \re{bigsum}_{2,2}\nonumber\\
&&+\tilde \Phi_{2,2}(z,\bar z,u)-\Phi_{2,2}(z,\bar z, u)+Q(f_2,\bar f_2)\nonumber\\
&& +\imag\left(iD_ug_0(u)\Phi_{2,2}+\frac{1}{2}D_u^2g_0(u)(2\Phi_{1,1}Q+\Phi_{1,1}^2)\right)\label{matrix22}\\
L_{3,3}(f_1, g_0)&=&Q(iD^2_uf_1(Q+\Phi_{1,1})^2)),\bar f_1) +Q(f_1, -iD^2_u\bar f_1(Q+\Phi_{1,1})^2)\nonumber\\
&&+2\real\left(Q(iD_uf_1(u)\Phi_{2,2}, \bar z)+\frac{1}{2}Q\left(\frac{1}{2}D_u^2f_1(u)(2\Phi_{1,1}Q+\{\Phi^2\}_{2,2}),  \bar z\right)\right)\nonumber\\
&&+\imag\left(iD_ug_0(u)\Phi_{3,3}+\frac{1}{2}D_u^2g_0(u)(2\Phi_{2,2}Q+\{\Phi^2\}_{3,3})\right.\nonumber\\
&& \left.-\frac{i}{6}D_u^3g_0(u)(3\Phi_{1,1}^2Q+\Phi_{1,1}^3+3\Phi_{1,1}Q^2)\right)\nonumber\\
&&+\tilde \Phi_{3,3}(z,\bar z,u)-\Phi_{3,3}(z,\bar z, u)\nonumber\\
&&+ \re{bigsum}_{3,3}\label{matrix33}
\end{eqnarray}
Furthermore, equation \re{21} for $p=2,3$ reads~:
\begin{equation}
\begin{aligned}
L_{2,1}(f_2) &= \real \left(D_ug_{0}(u)\right)\Phi_{2,1}
+Q(f_2,\bar f_1) 
+\tilde \Phi_{2,1}(z,\bar z, u)-\Phi_{2,1}(z,\bar z, u)+T_{2,1} \\
L_{3,1}(f_3) &= \real \left(D_ug_{0}(u)\right)\Phi_{3,1}
+Q(f_3,\bar f_1)
+\tilde \Phi_{3,1}(z,\bar z, u)-\Phi_{3,1}(z,\bar z, u)+T_{3,1}
\end{aligned}
\label{matrix21} 
\end{equation}

Let us recall that the operator $\Delta$ is given by 
$\Delta\colon \fpstwo{\mathcal{R}_{p,q}}{u}\to \fpstwo{\mathcal{R}_{p+1,q+1} }{u}  $, $\Delta R (u) = D_u R (u).Q(z,\bar z)$.
Then we have
\begin{equation}
    \label{L1}
    L_{1}(f_1,\real(g_0))=\left(
    \begin{matrix}
    \Delta \real (g_0) -2 \real Q(f_1,\bar z)\\
    - 2 \imag Q(\Delta f_1,\bar z)\\
    -\tfrac{1}{6} \Delta^3 \real (g_0) +  \real Q(\Delta^2 f_1,\bar z)
  \end{matrix}\right).
\end{equation}
Let us write
\begin{equation}
    \label{L2}
    L_{2}(f_2,f_3)= \begin{pmatrix}
    -Q(f_2, \bar z) \\
    -Q(f_3, \bar z)  
    \end{pmatrix}
\end{equation}
The system \eqref{matrix11}--\eqref{matrix21} now reads 
\begin{equation}\label{equ1}
L(f_1, f_2,f_3,\real(g_0))= {\mathcal G}(u,D_u^if_1, D_u^j\real(g_0), D_u^lf_2, \Phi_{123})
\end{equation}
where the indices ranges are:~$0\leq i\leq 2$, $0\leq j\leq 3$, 
and $0\leq l\leq 1$. Also, $\Phi_{123}$ stands for $(\Phi_{1,1},\Phi_{2,2},\Phi_{3,3},\Phi_{2,1},\Phi_{3,1})$. 
Let us emphasize the dependence of ${\mathcal G}$ on $\Phi_{123}$ below. We have 
\begin{equation}\label{equ1-modif}
{\mathcal G}=-(I-D_u\real(g_0))\Phi_{123}+\tilde {\mathcal G}(u,D_u^if_1, D_u^j\real(g_0), D_u^kg_1, D_u^lf_2, \Phi_{123})
\end{equation}
where $D_u\real(g_0)\Phi_{123}$ stands for  
$$
(D_u\real(g_0)\Phi_{1,1},D_u\real(g_0)\Phi_{2,2},D_u\real(g_0)\Phi_{3,3},D_u\real(g_0)\Phi_{2,1},D_u\real(g_0)\Phi_{3,1}).
$$
Furthermore, among $\Phi_{123}$, the $(i,j)$-component of $\tilde {\mathcal G}$ depends only on $\Phi_{\leq i-1,\leq j-1}$.

Here, ${\mathcal  G}$ is analytic in $u$ 
in a neighborhood of the origin, polynomial in its other arguments and 
\begin{equation}
    \label{L}
    L(f_1, f_2,f_3,\real(g_0))=\left(
    \begin{matrix}
    L_{1}(f_1,\real(g_0))\\
    L_{2}(f_2,f_3)
  \end{matrix}\right).
\end{equation}

The linear operator $L_1$ is defined from $(\real(g_0),f_1)\in \Bbb R\{u\}^d\times \Bbb C\{u\}^{n^2}\cong \Bbb R\{u\}^{k_3 + k_1}$ to ${\mathcal R}_{1,1}\{u\}\oplus {\mathcal R}_{2,2}\{u\}\oplus{\mathcal R}_{3,3}\{u\}\cong \Bbb R\{u\}^N$ for some $N$. 
The linear operator $L_2$ is defined from $(f_2,f_3)\in \Bbb C\{u\}^{n\binom{n+1}{2}} \times \C\{u\}^{n\binom{n+2}{3}} 
\cong \Bbb R\{u\}^{k_2+ k_4}$ to ${\mathcal R}_{2,1}\{u\} \times {\mathcal R}_{3,1} \cong \Bbb R\{u\}^M$ for some $M$. Each of these spaces is endowed with the (modified) Fisher scalar product of $\Bbb R\{u\}$. Here we have set~:
\begin{equation}\label{dimen}
k_1:=2n^2,\quad k_2:=2n\binom{n+1}{2},\quad  k_3:=d, \quad k_4:=2n\binom{n+2}{3}.
\end{equation}

Let ${\mathcal N}_{1}$ (resp.  ${\mathcal N}_{2}$) be the orthogonal subspace to the image of $L_{1}$ (resp. $L_{2}$) with respect to that scalar product~:
\begin{eqnarray}
{\mathcal R}_{1,1}\{u\}\oplus {\mathcal R}_{2,2}\{u\}\oplus{\mathcal R}_{3,3}\{u\}&=&\imag(L_{1})\oplus^{\bot} {\mathcal N}_{1}\nonumber\\
{\mathcal R}_{2,1}\{u\}\oplus {\mathcal R}_{3,1}\{u\} &=&\imag(L_{2})\oplus^{\bot} {\mathcal N}_{2}.\label{nf-spec}
\end{eqnarray}
These are the spaces of {\it normal forms} and they are defined to be the kernels of the adjoint operator with respect to the modified Fischer scalar product~: ${\mathcal N}_{1}=\ker L_{1}^*$, ${\mathcal N}_{2}=\ker L_{2}^*$; in terms of the normal form spaces introduced in \cref{sec:the_normalization_conditions}, we have in a natural way ${\mathcal N}_{1} \cong {\mathcal N}^{1}$ and ${\mathcal N}_{2} \cong {\mathcal N}^2_3$. Let $\pi_{i}$ be the orthogonal projection onto the range of $L_{i}$ and $\pi:=\pi_1\oplus\pi_2$.

The set of the seven previous equations encoded in \re{equ1} has the seven real unknowns $\real(f_1), \imag(f_1), \real(f_2), \imag(f_2), \real{(f_3)}, \imag{(f_3)}, \real(g_0)$. 


Let us project \re{equ1} onto the kernel of $L^*$, which is orthogonal to the image of $L$
with respect to the Fischer inner product, i.e. we impose the normal form conditions 
\eqref{e:defineNf}. 

 Since $\Phi_{123}$ belongs to that space, we have
$$
0=-(I-(I-\pi)D_u\real(g_0))\Phi_{123}+(I-\pi)\tilde {\mathcal G}(u,D_u^if_1, D_u^j\real(g_0), D_u^lf_2, \Phi_{123}).
$$
In other words, we have obtained
\begin{equation}\label{Phi123}
\Phi_{123} =  \left((I-(I-\pi)D_u\real(g_0)\right)^{-1}(I-\pi)\tilde {\mathcal G}(u,D_u^if_1, D_u^j\real(g_0), D_u^lf_2, \Phi_{123}).
\end{equation}
According to the triangular property mentioned above, we can express successively $\Phi_{1,1},\cdots,\Phi_{3,3}$ as an analytic function of only $u,D_u^if_1, D_u^j\real(g_0), D_u^lf_2$. 
Substituting in \re{equ1} and projecting down onto the image of $L$, we obtain
\begin{equation}\label{equ2}
L(f_1, f_2,f_3, \real(g_0))= \pi{\mathcal F}(u,D_u^if_1, D_u^j\real(g_0), D_u^lf_2,f_3)
\end{equation}
The equations corresponding to $L_2$ then turn into a set of implicit equations for $f_2$ and 
$f_3$, which we can solve uniquely in terms 
of $f_1$ and $\real g_0$. After substituting those 
solutions back into ${\mathcal{F}}$, we satisfy the normalization conditions in $\mathcal{N}_2$, 
and 
we turn up with a set of equations for $f_1$ and $\real g_0$: 
\begin{equation}\label{equ2a}
L_1(f_1,  \real(g_0))= \pi_1 {\mathcal F}_1(u,D_u^if_1, D_u^j\real(g_0))
\end{equation}
where the indices ranges are:~$0\leq i\leq 2$, $0\leq j\leq 3$, 
 and $0\leq l\leq 1$. Here, ${\mathcal F}_1$ denotes an analytic function of its arguments at the origin.

From now on, $\ord_0 f$ will denote the order of $f(z,\bar z,u)$ w.r.t $u$ at $u=0$. Let us recall that we always have
\begin{eqnarray}
\ord_0\tilde \Phi_{1,1}&\geq& 1\label{ord11}
\end{eqnarray}
We now claim that there is an analytic change of coordinates $z=z^*+f_1(z^*,w^*)+f_2(z^*,w^*) + f_3(z^*,w^*)$, $w=w^*+ g_0(w^*)$ such that also the diagonal terms of the new equation of the manifold are in normal form, that is $(\Phi_{1,1},\Phi_{2,2},\Phi_{3,3},\Phi_{2,1},\Phi_{3,1})\in\mathcal N_{1}\times \mathcal N_{2}$.  
In fact, we shall prove that there is exists a unique $(f_1,\real(g_0))\in \imag(L_{1}^*)$ with this property; if we would like 
to have {\em all} solutions to that problem, we will see that we can construct a unique solution for any given ``initial data'' in $\ker L_1$. 
Instead of working directly on equation \re{equ2a}, we shall first ``homogenize'' the derivatives of that system. By this we mean, that we apply operator $\Delta^2$ to the first coordinate of \re{equ2a} and  $\Delta$ to the second coordinate of \re{equ2a}. The resulting system reads
\begin{equation}\label{equ1b}
\tilde L_1( \tilde f_1,\real(\tilde g_0))= \tilde {\mathcal F}_1(u,D_u^i \tilde f_1, D_u^j\real(\tilde g_0))
\end{equation}
where 
\begin{equation}
    \label{L1b}
  \tilde L_{1}(\tilde f_1,\real(\tilde g_0))=\left(
    \begin{matrix}
    \Delta^3 \real (\tilde g_0) -2 \real Q(\Delta^2 \tilde f_1,\bar z)\\
    - 2 \imag Q(\Delta^2 \tilde  f_1,\bar z)\\
    -\tfrac{1}{6} \Delta^3 \real (\tilde g_0) +  \real Q(\Delta^2 \tilde  f_1,\bar z).
  \end{matrix}\right)=:{\mathcal L}_1(D_u^2\tilde f_1,D_u^3\real(\tilde g_0))
\end{equation}

Here, $\mathcal L_1$   denotes a  linear 
operator on the finite dimensional vector spaces ${\rm{Sym}}^2 (\Cd, \Cn) \times {\rm{Sym}}^3(\Cd, \Rd)$ 
, and we  have set $f_1 = j^1 f_1 + \tilde f_1$, $ g_0 = j^2 g_0 +\tilde g_0$, and 
$$
\tilde L_1:=\tilde{\mathcal D}\circ L_1,\quad \tilde{\mathcal F}_1 (u,D_u^i \tilde f_1, D_u^j\real(\tilde g_0)):= \tilde{\mathcal D}\circ \pi_1\circ{\mathcal F}_1 ( u, D_u^i f_1, D_u^j g_0 ) ,
$$
where
$$
\tilde{\mathcal D}:= \left(\begin{matrix}\Delta^2& 0&0\\ 0&\Delta&0\\ 0&0&I\end{matrix}\right).
$$

 Using the right hand side of \re{matrix11}, \re{matrix22}, \re{matrix33}, 
 and differentiating accordingly, 
 we see  that $\text{ord}_0(\tilde{\mathcal F}(u,0))\geq 1$. 

Let us set ${\bf m}=(m_1, m_3) = (2,3)$ and ${\mathcal F}_{2,\bf m}^{\geq 0} :=\left(\mathbb A_d^{k_1}\right)_{\geq m_{1}} \times \left(\mathbb A_d^{k_3}\right)_{\geq m_{3}}$ where the $k_i$'s are defined in \re{dimen}. Then a tuple of analytic functions 
$$
H:=(H_1,H_3)=(\tilde f_1, \real(\tilde g_0))
$$ 
with $\ord_0 f_1 \geq 2$, $\ord_0 g_0 \geq 3$ is an element of ${\mathcal F}_{2,\bf m}^{\geq 0}$. 
Then, equation \re{equ1b} reads~: 
\begin{eqnarray}
\mathcal S(H) & = &\tilde{\mathcal F}(u, j^{\bf m}_u H)\label{equ2solve}\\
{\mathcal S}(H) &:= &{\mathcal L}_1(D_u^2H_1,D_u^3H_3).\label{linear-final}
\end{eqnarray}
Let us show that the assumptions of the Big denominators theorem \ref{main-thm-BD} are satisfied. First of all, for any integer $i$, let us set $H^{(i)}:= (H_1^{(m_1+i)}, H_3^{(m_3+i)})$. Their linear span will be denotes by ${\cal H}^{(i)}$. Then, for any $i$, 
${\mathcal S}(H^{(i)})$ is homogeneous of degree of degree $i$. Let us consider the linear operator $d: (\tilde f_1,\real (\tilde g_0))\mapsto (D_u^2\tilde f_1,D_u^3\real(\tilde g_0))$. It is one-to-one from ${\mathcal F}_{2,\bf m}^{\geq 0}$ and onto the space of ${\rm{Sym}}^2 (\Cd, \Cn) \times {\rm{Sym}}^3(\Cd, \Rd)$-valued analytic functions in $(\Bbb R^d,0)$. Let $V\in \image({\cal S})$. We recall that ${\cal S}={\cal L}_1\circ d$. Let us set $K:=({\cal L}_1{\cal L}_1^*)^{-1}(V)$. It is well defined since $V$ is valued in the range of ${\cal L}_1$.
Therefore, $\|K\|\leq \alpha \|V\|$ for some positive number $\alpha$. On the other hand, we have 
$ {\cal L}_1^*K \in \image d $, so we can 
(uniquely) solve the equation 
$$
d(\tilde f_1, \real(\tilde g_0))= {\cal L}_1^*K.
$$
This solution now satisfies clearly~: 
\begin{eqnarray*}
\|\tilde f_1^{(i)}\|&\leq & \frac{|||{\cal L}_1^*|||\alpha}{i^2}\|V^{(i)}\|\\
\|\real(\tilde g_0^{(i)})\|&\leq & \frac{|||{\cal L}_1^*|||\alpha}{i^3}\|V^{(i)}\|\\
{\cal S} (\tilde f_1, \real(\tilde g_0))& = & {\cal L}_1d(\tilde f_1, \real(\tilde g_0))= {\cal L}_1{\cal L}_1^*K=V.
\end{eqnarray*}
Hence, ${\cal S}$ satisfies the Big Denominators property with respect to ${\bf m}=(m_1, m_3) = (2,3)$.

On the other hand, let us show that $\tilde{\mathcal F}(u, j^{\bf m}_mH)$ strictly increases the degree by $q=0$. This means that 
$$
\text{ord}_0\left(\tilde{\mathcal F}(u, j^{\bf m}_mH)-\tilde{\mathcal F}(u, j^{\bf m}_m\tilde H)\right)> \text{ord}_0(H-\tilde H).
$$
According to Corollary \ref{q=0} of Appendix \ref{sec-PDEs}, we just need to check that the system is regular. 

%

So let us now prove that the analytic differential map $\tilde{\mathcal F}(u,j_u^{\bf m})$ is {\it regular} in the sense of definition \ref{def-nonlin-bd}. To do so, we have to differentiate each term of $\tilde{\mathcal F}(u,j_u^{\bf m})$ with respect to the unknowns and their derivatives and show that the vanishing order of the functions their multiplied by are greater or equal than number $p_{j,|\alpha|}$ as defined in \re{p-reg} in definition \ref{def-nonlin-bd}. We recall that $q=0$. Therefore, these number are either $0$ (no condition) or $1$ (vanishing condition). The later correspond to the vanishing at $u=0$ of the coeffcient in front the highest derivative order of the unknown~:
$$
\frac{\partial \tilde{\mathcal F}_i}{\partial u_{j,\alpha}}(u,\partial H),\quad |\alpha|=m_j.
$$
where $H=(H_1,\ldots, H_r)\in \widehat\cF_{r,\bf m}^{\geq 0}$. 

But this condition in turn is automatically fulfilled by the construction of 
the system, since we have put exactly the highest order derivatives appearing in each of the
conjugacy equations appearing with a coefficient which 
is nonzero when evaluated at $0$ into the linear part of the operator, and no of the 
operations which we applied to the system changes this appearance. Let 
us recall that $f_1 (0) = \real g(0) = 0$. 
As a conclusion, we see that the map $\tilde{\mathcal F}(u,j_u^{\bf m})$ is {\it regular}. Furthermore, according to \re{linear-final}, the linear operator $\mathcal S$ has the Big Denominator property of order ${\bf m}=(2,1,3)$. 
Then according the Big Denominator theorem \ref{main-thm-BD} with $q=0$, equation \re{equ2solve} has a unique solution $H^{\geq 0}\in {\mathcal F}_{2,\bf m}^{\geq 0} :=\left(\mathbb A_d^{k_1}\right)_{\geq m_{1}} \times \left(\mathbb A_d^{k_3}\right)_{\geq m_{3}}$. This provides the terms of higher
order in the expansions of $f_1$ and $\real g_0$, and therefore, we proved the 
\begin{proposition}
There is exists a unique analytic map $(f_1,\real(g_0),f_2,f_3)\in \imag(L_{1}^*)\times \imag(L_{2}^*)$ 
such that under the change of coordinates $z=z^*+f_1(z^*,w^*)+f_2(z^*,w^*)+f_3(z^*, w^*)$, $w=w^*+ g_0(w^*)$, the $(1,1)$, $(2,1)$, $(2,2)$  and $(3,3)$ terms of the new equation of the manifold are in normal form, that is, $\Phi\in\mathcal{N}^0 \cap \mathcal{N}^d  \cap \mathcal{N}^1_{\leq 3}$ as defined in \cref{sec:the_normalization_conditions}. 
\end{proposition}

\subsection{Normalization of terms \texorpdfstring{$(m,1)$, $m\geq 4$}{(m,1), m>=4}}

Let us perform another change of coordinates of the form $z=z^*+\sum_{p\geq 4}f_p (z^*,w^*)$, $w=w^*$. According to \re{big-equiv}
we obtain by extracting the $(p, 1)$-terms, $p\geq 4$
\begin{equation}\label{equ-p1}
-Q(f(z,u),\bar z)=\tilde \Phi_{*,1}(z+f(z,u),\bar z, u)-\Phi_{*,1}(z,u),
\end{equation}
where $\tilde \Phi_{*,1}(z,\bar z,u):=\sum_{p\geq 4}\tilde \Phi_{p,1}(z,\bar z,u)$ is analytic at $0$. We recall that $\tilde \Phi(z,0,u)=\tilde \Phi(0,\bar z,u)=0$. Therefore, by Taylor expanding, we obtain
\begin{eqnarray*}
\{\tilde \Phi_{\geq 3}\left(f,\bar f, u\right)\}_{*,1}&=& \left\{\tilde \Phi_{\geq 3}\left(z+f_{\geq 2}(z,u),\bar z, u\right)\right.\\
&&+\frac{\partial \tilde \Phi_{\geq 3}}{\partial z}(f_{\geq 2}(z,u+iQ+i\Phi)-f_{\geq 2}(z,u))\\
&&+\left.\frac{\partial \tilde \Phi_{\geq 3}}{\partial \bar z}\bar f_{\geq 2}(\bar z,u-iQ-i\Phi)+\cdots\right\}_{*,1}\\
\end{eqnarray*}
Since $\tilde \Phi_{p,0}=0$ for all integer $p$, the previous equality reads
$$
\{\tilde \Phi_{\geq 3}\left(f,\bar f, u\right)\}_{*,1}= \tilde \Phi_{*,1}\left(z+f_{\geq 2}(z,u),\bar z, u\right).
$$
\subsubsection{A linear map} 
\label{sec:a_linearized_map}
In this section we  consider the linear map $\mathcal{K}$, which maps a germ of holomorphic function $f(z)$  at the origin to
\begin{equation}\label{e:defK}
  \mathcal{K} (f) = Q(f(z),\bar z).
\end{equation}
This complex linear operator $\mathcal{K}$ is valued in the space of power series in $z,\bar z$, valued in $\Cd$ which are linear in $\bar z$. We will first restrict $\mathcal{K}$ to a map $\mathcal{K}_m$ on 
the space of homogeneous polynomials of degree $m$ in $z$, with values in $\Cn$, 
For any $C,\delta>0$, let us define the Banach space
\begin{equation}\label{elem-banach}
{\mathcal B}_{n,C,\delta}:=\{f=\sum_m f_m, f_m\in {\mathcal H}_{n,m}, \|f_m\|\leq C \delta^m\}.
\end{equation}
Then, the map $\mathcal{K}_m$ is valued in the space $\mathcal{R}_{m,1}$ of polynomials in $z$ and $\bar z$, valued in $\Cd$, which are linear in $\bar z$ and homogeneous of degree $m$ in $z$. 
Let us consider the space ${\mathcal R}_{*,1}:=\bigoplus_m{\mathcal R}_{m,1}$ as well as
$$
\{f=\sum_m f_m\in {\mathcal R}_{*,1}, \|f_m\|\leq C\delta^m\}
$$ 
where $\|.\|$ denotes the modified Fischer norm and $C,\delta$ a positive numbers. The latter is a Banach space denoted ${\mathcal R}_{*,1}(C,\delta)$.

In particular, let us note that 
if we write $P_k = \sum_j P_k^j(z) \bar z_j$ with $P_k^j \in \mathcal{H}_m$, then
\begin{equation}\label{e:normslinear}
  \vnorm{P_k^{\phantom{j}}}^2 = (m+1)\sum_{j=1}^n \vnorm{P_k^j}^2.
\end{equation}
Let us write $P_k = \bar z^t \mathbf{P}_k$ where $\mathbf{P}_k= (P_k^1,\dots,P_k^n)^t$.
We can now formulate
\begin{lemma}\label{lem:constantbound}
  There exists  a constant $C>0$ such that for all $m \geq 0$, we have that 
  \[ \vnorm{f} \leq \frac{C}{\sqrt{(m+1)}} \vnorm{\mathcal{K}_m f}. \]
In particular, $\mathcal K$ has a bounded inverse on its image : if $g\in {\mathcal R}_{*,1}(M,\delta)\cap Im \mathcal K$, then ${\mathcal K}^{-1}(g)\in {\mathcal B}_{M,\delta}$ and
$$
\|{\mathcal K}^{-1}(g)\|\leq C \|g\|.
$$
\end{lemma}
\begin{proof}We consider the $n\times (nd)$-matrix  $J$ defined by 
  \begin{equation}
    \label{e:defJ} J = \begin{pmatrix} J_1 \\ \vdots \\J_d
    \end{pmatrix} .
  \end{equation}
  Since $\inp{\cdot}{\cdot}$ is nondegenerate, we can choose an invertible
   $n\times n$-submatrix $\tilde J$ from $J$, composed of the rows in the spots $(j_1,\dots, j_n)$; let 
  $k(j_\ell)$ denote which $J_k$ the row $j_\ell$ belongs to. Then, if $\mathcal{K}_m f = P$, we have for every $k = 1, \dots ,d$ that 
  $\bar z^t J_k f = \bar z^t \mathbf{P}_k$. Hence, by complexification we see that $J_k f = \mathbf{P}_k$.
  
  Let $\tilde P= (P^{j_1}_{k(j_1)} , \dots P^{j_n}_{k(j_n)})^t$. Then $\tilde J f = \tilde P$, and we can write 
  $f = (\tilde J)^{-1} \tilde P$. Hence, 
  \[ \vnorm{f}^2 \leq C \sum_{\ell=1}^n \vnorm{P_{k(j_\ell)}^{j_\ell}}^2 \leq \frac{C}{m+1} \vnorm{P}^2, \]
  by the observation in \eqref{e:normslinear}.
\end{proof}

In order to find an explicit complementary space to $\image \mathcal{K}_m$, we will use the Fischer inner product 
to compute its adjoint $\mathcal{K}_m^*$. We first note, that since the components of $\mathcal{R}_{m,1}$ are 
orthogonal to one another, if we write $\mathcal{K}_m = (\mathcal{K}_m^1,\dots ,\mathcal{K}_m^d)$, then 
$\mathcal{K}_m^* = (\mathcal{K}_m^1)^*+ \dots +(\mathcal{K}_m^d)^*$. The adjoints of  the maps 
$\mathcal{K}_m^k$, $k=1,\dots,d$   are computed via
\begin{equation}
  \label{e:adjoint1} \begin{aligned}
    \inp{\mathcal{K}_m^k f}{P_k} & = \inp{ {\bar z}^t J_k f}{ \sum_j P_k^j {\bar z_j}}\\
    & = \inp{\sum_{p,q = 1}^n (J_k)^p_q \bar z_p f^q }{\sum_j P_k^j {\bar z_j}} \\
    & = \frac{1}{m+1} \sum_{p,q =1 }^n (J_k)^p_q \inp{ f^q }{ P_k^p } \\
    & = \frac{1}{m+1} \sum_{p,q=1}^n \inp{f^q}{ \overline{(J_k)^p_q} P_k^p}
  \end{aligned} 
\end{equation}
to be given by 
\begin{equation}
  \label{e:adjoint2} 
  (m+1) ((\mathcal{K}_m^k)^* P_k)^q = \sum_{p=1}^n  (J_k)_p^q P^p_k = \sum_{p=1}^n (J_k)_p^q \dop{\bar z_p} P_k,
\end{equation}
or in more compact notation,
\begin{equation}
  \label{e:adjoint3}  (m+1) (\mathcal{K}_m^k)^* P_k =  \left(J_k  \dop{\bar z}\right) P_k.
\end{equation}

We now define the subspace $\mathcal{N}^1_{m,1}$ to consist of the elements of the kernel of $\mathcal{K}_m^*$, i.e.
  \begin{equation}\label{e:defNm1}
    \mathcal{N}^1_{m,1} := \left\{ P=(P_1, \dots P_d)^t\in\mathcal{R}_{m,1} \colon
  \sum_{k=1}^d  \left(J_k \dop{\bar z}\right) P_k = 
  \sum_{k=1}^d  J_k  \mathbf{P}_k = 0 \right\}.
  \end{equation}

\begin{proposition}
There exists a holomorphic transformation $z=z^*+f_{\geq 4}(z,w)$, $w=w^*$ such that, the new equation of the manifold satisfies
$$
\Phi_{p,1}\in {\mathcal N}_{p,1},\quad p\geq 4.
$$
\end{proposition}
\begin{proof}
Let $\pi_{*,1}$ be the orthogonal projection onto the range of ${\mathcal K}$. Then since we want $\Phi_{*,1}$ to belong the normal forms space ${\mathcal N}^1_{*,1}$, we have to solve
$$
-{\mathcal K}(f):=-Q(f(z,u),\bar z)=\pi_{*,1} \tilde \Phi_{*,1}(z+f(z,u),\bar z, u).
$$
According to \rl{lem:constantbound}, the latter has an analytic solution by the implicit function theorem and we are done.
\end{proof}

\section{Convergence of the formal normal form} 
\label{sec:preparing_the_conjugacy_equation}
We are now going to prove convergence of the formal 
normal form in \cref{sec:a_full_formal_normal_form_proof_of_cref} 
under the additional condition 
of \cref{thm:main3} on the 
formal normal form. The goal of this section is to 
show that one can, under this additional condition, 
replace the nonlinear terms in the conjugacy equations 
for the terms of order up to $(3,3)$, by another 
system which allows for the application of the 
big denominator theorem.

We are again going to consider two real-analytic Levi-nondegenerate
submanifolds of $\CN$, but we now need to use their {\em complex defining equations} 
$ w = \theta (z, \bar z, \bar w) $ and $ w = \tilde \theta (z, \bar z, \bar w)$, 
respectively, where $\theta$ and $\tilde \theta$ are
germs of analytic maps at the origin in $\Cn\times\Cn\times \Cd$ valued in $\Cd$; analogously to the real defining functions, 
we think about $\tilde \theta$ as the ``old'' and about $\theta$ as 
the ``new'' defining equation. 

When dealing with the 
complex defining function, we will usually write 
$\chi = \bar z$  and $\tau = \bar w$. Recall that a map $\theta \colon \C^{2n + d} \to \Cd $  
determines a real submanifold if and only if 
the reality relation
\begin{equation}
  \label{e:reality} \tau = \theta (z, \chi, \bar \theta (\chi,z,\tau))  
\end{equation}
holds. $\theta$ is obtained from a 
real defining equation $\imag w = \varphi (z, \bar z , \real w)$ by solving the 
equation 
\[ \frac{w - \bar w}{2i} = \varphi \left( z, \bar z, \frac{w + \bar w}{2}\right) \]
for $w$.

We will already at the outset prepare our conjugacy equation so 
that $(z,w)$ are {\em normal coordinates}
for these submanifolds, i.e. that $\theta (z, 0, \tau  ) = \theta (0, \chi, \tau ) = \tau$
and we assume that $\tilde \theta (z', 0, \tau'  ) =  \tilde \theta (0, \chi', \tau' ) = \tau'$. In terms of the  original
``real'' defining function this means $\varphi(z, 0 , s) = \tilde \varphi (z,0,s) = 0$ (and analogously for $\tilde \varphi$). 

If our real real defining function, as assumed
before, satisfies $\varphi(z,\bar z, s) = Q(z,\bar z) +  \Phi (z, \bar z ,s) $, we can  write 
\[ \theta (z, \chi, \tau ) =  \tau + 2 iQ(z,\chi)  + S (z,\chi,\tau).  \]

$S$ can be further decomposed as 
\[ S (z, \chi, \tau) = \sum_{j,k=1}^\infty S_{j,k} (\tau) z^j \chi^k. \]
Here we think of $S_{j,k}$ as a power series in $\tau$ taking values in 
the space of multilinear maps on $(\Cn)^{j+k}$ which are symmetric in 
their first $j$ and in their last $k$ variables separately, taking values in $\Cd$ (i.e.
polynomials in $z$ and $\chi$ homogeneous of degree $j$ in $z$ and of degree $k$ in $\chi$), 
and for any such map $L$, write $L z^j \chi^k $ for $L(\underbrace{z,\dots, z}_{j \text{times}}, \underbrace{\chi, \dots ,\chi}_{k \text{times}} )$. 

We note for future reference the following simple observations: 
\begin{equation}
\label{eqn:simplerelphiS1}
\begin{aligned}
S_{1,\ell} = 2 i \Phi_{1,\ell} , \quad S_{\ell, 1} = 2 i \Phi_{\ell, 1}, \qquad \ell \geq 1, \quad 
S_{2,2} = 2 i \left( \Phi_{2,2} + i \Phi_{1,1}' (Q + \Phi_{1,1}) \right) , \\
S_{2,3} = 2 i \left( \Phi_{2,3} + i \Phi_{1,2}' (Q + \Phi_{1,1}) + i \Phi_{1,1}' \Phi_{1,2} \right), \qquad 
S_{3,2} = 2 i \left( \Phi_{3,2} + i \Phi_{2,1}' (Q+ \Phi_{1,1} ) + i \Phi_{1,1}' \Phi_{2,1} \right)  .
\end{aligned} 
\end{equation}
and 
\begin{equation}
\label{eqn:simplerelphiS2}
\begin{aligned}
\Phi_{2,2} &= \frac{1}{2i} S_{2,2} - \frac{1}{4i} S_{1,1}'(2 i Q + S_{1,1}) \\
\Phi_{2,3} &= \frac{1}{2i} S_{2,3} - \frac{1}{4i} S_{1,1}' S_{1,2} - \frac{1}{4i} S_{1,2}'(2 i Q +  S_{1,1}) , \\
\Phi_{3,2} &= \frac{1}{2i} S_{3,2} - \frac{1}{4i} S_{1,1}' S_{2,1} - \frac{1}{4i} S_{2,1}'(2 i Q+ S_{1,1})  \\
\Phi_{3,3} &= \frac{1}{2i} S_{3,3} - \frac{1}{4i} S_{2,2}' (2 i Q + S_{1,1}) - \frac{1}{8i} S_{1,1}' (2 S_{2,2} + S_{1,1}' (2i Q + S_{1,1})) +  \\ &- \frac{1}{4i} S_{1,2}' S_{2,1} - \frac{1}{4i} S_{2,1}' S_{1,2} + \frac{1}{16i} S_{1,1}'' (2i Q + S_{1,1})^2 .
\end{aligned} 
\end{equation}
Furthermore, from the fact that $\theta (z, \chi, \bar \theta (\chi, z, w)) = w$, we obtain 
the following equations relating $S_{j,k}$ and their conjugates:
\begin{equation}
\label{eqn:realrelS1} 
S_{1 , \ell} (w)+ \bar S_{\ell, 1} (w)   = 0, \quad S_{2,2} - S_{1,1}'(2iQ - \bar S_{1,1})  + \bar S_{2,2}= 0, \quad S_{2,3} -  S_{1,2}' (2i Q - \bar S_{1,1}) +S_{1,1}' \bar S_{2,1} + \bar S_{3,2} =0 
\end{equation}

A map $H=(f,g)$ maps the manifold defined by $w = \theta (z,\bar z, \bar w)$
into the one defined by $w'= \tilde \theta (z',\bar z', \bar w')$  if and only if  the following equation 
is satisfied: 
\begin{equation}
\label{e:mapping1} g(z,\theta(z,\chi,\tau)) = \tilde \theta (f (z,\theta (z,\chi, \tau))
, \bar f (\chi ,\tau) , \bar g(\chi, \tau)).
\end{equation}
An equivalent equation is (after application of \eqref{e:reality})
\begin{equation}
\label{e:mapping2} g(z,w) = \tilde \theta (f (z,w)
, \bar f (\chi ,\bar \theta (\chi,z,w)) , \bar g(\chi, \bar \theta (\chi,z,w))).
\end{equation}
If we set $\chi = 0$ in \eqref{e:mapping2}, the assumed normality of the 
coordinates, i.e. the equation $\theta(z, 0 , w) = 0$, is equivalent
 $g(z,w) = \tilde \theta (f(z,w), \bar f(0,w) , \bar g(0,w)) $; in particular,
 for $w = \theta (z,\chi,\tau)$, we have the (also equivalent) condition
 \begin{equation}
  \label{e:mapping1a} g(z,\theta(z,\chi,\tau)) = \tilde \theta (f (z,\theta (z,\chi, \tau))
, \bar f (0,\theta (z,\chi, \tau)) , \bar g(0,\theta (z,\chi, \tau))). 
 \end{equation}
 On the other hand setting $z=0$, 
 observing $\theta (0, \chi,\tau) = \tau$,
 and using (the conjugate of) \eqref{e:mapping1} 
 we also have 
 \begin{equation}
 \bar g (\chi, \tau ) = \bar{\tilde\theta} (\bar f (\chi,\tau), f(0,\tau),g(0,\tau))
 \end{equation}
Combining this with \eqref{e:mapping1}, we obtain 
the following equivalent equation, which now guarantees the normality of $(z,w)$: 
\begin{equation}
\begin{aligned}
  \label{e:mapping3}
  \tilde \theta &(f (z,\theta (z,\chi, \tau))
, \bar f (0,\theta (z,\chi, \tau)) , \bar g(0,\theta (z,\chi, \tau))) \\& = \tilde \theta \left(f (z,\theta (z,\chi, \tau))
, \bar f (\chi ,\tau) , \bar{\tilde\theta} (\bar f (\chi,\tau), f(0,\tau),g(0,\tau))\right).
\end{aligned}
\end{equation}
Lastly, we can use one of the equations implicit in 
\eqref{e:mapping3} to eliminate $\imag g$ from it. 
This is easiest done using \eqref{e:phi00prep}, which (after extending to complex
$w$) becomes 
\begin{equation}
(\imag g) (0,w) = \tilde \varphi(f(0,w), \bar f (0,w), (\real g) (0,w)).\label{e:imaggeqn}
\end{equation}
Substituting this relation into \eqref{e:mapping3} 
eliminates the dependence on $\imag g$ completely 
from the equation, only $\real g$ appears now. 

We now substitute $f = z + f_{\geq 2} (z,w )$, where $f$ only contains terms of quasihomogeneity greater than $1$, and 
write 
\[ f_{\geq 2} (z,w) = \sum_{k\geq 0} f_k (w) z^k, \quad 
g(0,w) = w + g_0 (w); \]
we also write $\psi= \real g_0$ for brevity. Let us first disentangle 
the equation \eqref{e:imaggeqn}. In our current notation, this 
reads 
\begin{equation}
\label{e:imaggeqn2} (\imag g_0) (w) = \tilde \varphi (f_0 (w), \bar f_0 (w), w + \psi (w)).
\end{equation}
By virtue of the fact that $\tilde \varphi (z,0,s) = 0$, this exposes $\imag g_0$ as an nonlinear expression in $f_0$, $\bar f_0$ , and $\psi$. 

We can thus rewrite \eqref{e:mapping3} as 
\begin{equation}
\label{e:mapping3a}
\begin{aligned}
\tilde \theta &\left(z+ f_{\geq 2}
, \bar f_0\circ \theta  , \theta + \psi\circ \theta + i \tilde \varphi (f_0 \circ \theta, \bar f_0 \circ \theta, \theta + \psi \circ \theta) \right) \\& = 
\tilde \theta \left(z+ f_{\geq 2}, 
\chi + \bar f_{ \geq 2} , \bar{\tilde\theta} (\chi +  \bar f_{\geq 2}, f_0 ,\tau + \psi + i \tilde \varphi (f_0 (w), \bar f_0 (w), w + \psi (w)))\right),
\end{aligned}
\end{equation}
where we abbreviate $f_{\geq 2} = f_{\geq 2} (z, \theta (z,\chi, \tau))$ and $\bar f_{\geq 2} = \bar f_{\geq 2} (\chi,\tau)$.

We will now extract terms which are linear in the variables $f_{\geq 2}$, $\bar f_{\geq 2}$,  and $\psi$ from this equation. We rewrite:
\[ \begin{aligned}
\tilde \theta 
&\left(z+ f_{\geq 2}, \bar f_0\circ \theta  , \theta + \psi\circ \theta + i \tilde \varphi (f_0 \circ \theta, \bar f_0 \circ \theta, \theta + \psi \circ \theta) \right) \\& = 
\tau + 2 i Q (z,\chi) + S + \psi \circ \theta + 2 i Q (z, \bar f_0\circ \theta ) + \dots
 \end{aligned} \] 
 \[ \begin{aligned}
 \tilde \theta & \left(z+ f_{\geq 2}, 
, \chi + \bar f_{ \geq 2} , \bar{\tilde\theta} (\chi +  \bar f_{\geq 2}, f_0 ,\tau + \psi + i \tilde \varphi (f_0 (w), \bar f_0 (w), w + \psi (w)))\right) \\
& = \bar{\tilde\theta} (\chi +  \bar f_{\geq 2}, f_0 ,\tau + \psi + i \tilde \varphi (f_0 (w), \bar f_0 (w), w + \psi (w))) + 2 i Q (z+ f_{\geq 2}, 
, \chi + \bar f_{ \geq 2}) + \dots \\
& = \tau + 2 i Q(z,\chi) \psi - 2 i Q(f_0 , \chi ) + 2 i Q (z, \bar f_{\geq 2}) + 2 i 
Q(f_{\geq 2}, \chi)  + \dots,
 \end{aligned} \]
where we will elaborate on the 
terms which appear in the dots a bit below. 

 We can thus further express the conjugacy equation \eqref{e:mapping3a} in 
the following form: 
\begin{equation}
\label{e:mapping3b} \begin{aligned}
\psi\circ &\theta - \psi + 2 i  Q(z, \bar f_0 \circ \theta ) + 2 i Q (f_0 , \chi ) - 2 i Q (z, \bar f_{\geq 2}) - 2i Q (f_{\geq 2}, \chi)  \\ 
&= \tilde{\mathcal T} \left( z,\chi, \tau, f_0, \bar f_0, \psi , f_0 \circ \theta , 
\bar f_0 \circ \theta, \psi\circ \theta, f_{\geq 2}, \bar f_{\geq 2} \right) - S,
\end{aligned}
\end{equation}
where $\tilde{\mathcal T}$ has 
the property that in the further expansion to follow, 
it will only create ``nonlinear terms''. 

We now restrict \eqref{e:mapping3b}
to the space of space of power series which are homogeneous of degree up to at most $3$ in
$z$ and $\chi$. By replacing the compositions $\psi \circ \theta$, $\bar f_0 \circ \theta $, and  $f_j \circ \theta$, for $j\leq 3$, by their Taylor expansions, we get  
\[ \begin{aligned}
 \psi (\tau + 2iQ(z,\chi)+ S(z,\chi,\tau)) &= \sum_{k=0}^3 \psi^{(k)} (\tau) \left(2iQ(z,\chi)+ S(z,\chi,\tau) \right)^k, \quad \mod (z)^4 + (\chi)^4\\
 \bar f_0 (\tau + 2iQ(z,\chi)+ S(z,\chi,\tau)) &= \sum_{k=0}^3 \bar f_0^{(k)} (\tau) \left(2iQ(z,\chi)+ S(z,\chi,\tau) \right)^k,  \quad \mod (z)^4 + (\chi)^4\\
 f_j (\tau + 2iQ(z,\chi)+ S(z,\chi,\tau)) &= \sum_{k=0}^{3-j} \bar f_j^{(k)} (\tau) \left(2iQ(z,\chi)+ S(z,\chi,\tau) \right)^k,  \quad \mod (z)^4 + (\chi)^4
\end{aligned}   \]
the resulting equations, ordered by powers of $(z,\chi)$, writing $h=(f_0,\bar f_0, \psi)$, and saving space by setting
 $\varphi^{\leq j} = (\varphi, \varphi', \dots, \varphi^{(j)})$
 and 
 \[S^{< p, <q} = \left(S_{k,\ell} \colon k< p, \ell \leq q \text{ or } k\leq p, \ell<q   \right),\] become
\[ 
\begin{aligned}
&z \chi \qquad & -\psi' Q +  Q(z,\bar f_1 ) +  Q (f_1 ,\chi) &= 
\frac{S_{1,1}}{2i} + \tilde{\mathcal T}_{1,1} \left(h^{\leq 1}, f_1, \bar f_1 \right)\\
&z^2 \chi \qquad & - 2i Q( z, \bar f_0' Q) + Q(f_2, \chi) & = \frac{S_{2,1}}{2i} + \tilde{\mathcal T}_{2,1} \left(h^{\leq 1}, f_1^{\leq 1}, \bar f_1 , S_{1,1} \right)
\\
&z^3 \chi \qquad & Q(f_3, \chi) & = \frac{S_{3,1}}{2i} + \tilde{\mathcal T}_{3,1} \left(h^{\leq 1}, f_1^{\leq 1}, f_2,  \bar f_1 , S^{<3,<1} \right)
\\
&z \chi^2 \qquad &  2 iQ(f_0' Q, \chi) + Q(z, \bar f_2)  & = \frac{S_{1,2}}{2i} + \tilde{\mathcal T}_{1,2} \left(h^{\leq 1}, f_1, \bar f_1 , S_{1,1} \right)
\\
&z \chi^3 \qquad & Q(z, \bar f_3) & = \frac{S_{1,3}}{2i} + \tilde{\mathcal T}_{1,3} \left(h^{\leq 1}, f_1, f_2, \bar f_1 , S^{<1, <3}\right)\\
&z^2 \chi^2 \qquad & - i \psi'' Q^2 +2 i  Q (f_1' Q , \chi)  & = \frac{S_{2,2}}{2i} + \tilde{\mathcal T}_{2,2} \left(h^{\leq 2}, f_1^{\leq 1}, f_2, \bar f_1, \bar f_2 , S^{<2,<2} \right)\\
&z^2 \chi^3 \qquad &  -2 Q(f_0'' Q^2 , \chi)  & = \frac{S_{2,3}}{2i} + \tilde{\mathcal T}_{2,3} \left(h^{\leq 2}, f_1^{\leq 1}, f_2, \bar f_1 , \bar f_2 , S^{<2, <3} \right)  \\
&z^3 \chi^2 \qquad &  2i Q(f_2' Q , \chi) + 2 Q(z,\bar f_0'' Q^2) &= \frac{S_{3,2}}{2i} + \tilde{\mathcal T}_{3,2} \left(h^{\leq 2}, f_1^{\leq 2}, f_2, \bar f_1 , \bar f_2 , S^{<3,<2} \right) \\
&z^3 \chi^3 \qquad & \frac{2}{3} \psi'''Q^3 - 2 Q(f_1'' Q^2, \chi)  & = \frac{S_{3,3}}{2i } + \tilde{\mathcal T}_{3,3} \left(h^{\leq 3}, f_1^{\leq 2}, f_2^{\leq 1}, \bar f_1, \bar f_2 , S^{<3,<3} \right)
\end{aligned}
\]
The ``nonlinear terms'' $\tilde{\mathcal{T}}_{(p,q)}$ have the property that  the derivatives of highest order 
appearing in each line, if they appear in the nonlinear part, then 
their coefficient vanishes when evaluated at $\tau = 0$. (One can 
go through very similar arguments as in \cref{sec:initialquadric} 
to convince oneself of that fact).

This system has the problem that the equations for the $z^2 \chi $ and $z^3 \chi$ involve $f_1'$ 
and that the equation for $z^3 \chi^2$ inolves $f_1 ''$, 
which effectively turns the 
full system of 
equations {\em singular}: In order to see that, consider 
the last two lines of the preceding system, brought to the same order of 
differentiation: 
\[ \begin{aligned}
&z^3 \chi^2 \qquad &  2i Q(f_2'' Q^2 , \chi) + 2 Q(z,\bar f_0''' Q^3) &= \frac{S_{3,2}'Q}{2i} + \hat{\mathcal T}_{3,2} \left(h^{\leq 3}, f_1^{\leq 3}, f_2^{\leq 1}, \bar f_1^{\leq 1} , \bar f_2^\leq 1 , \hat S^{<3,<2} \right) \\
&z^3 \chi^3 \qquad & \frac{2}{3} \psi'''Q^3 - 2 Q(f_1'' Q^2, \chi)  & = \frac{S_{3,3}}{2i } + \tilde{\mathcal T}_{3,3} \left(h^{\leq 3}, f_1^{\leq 2}, f_2^{\leq 1}, \bar f_1, \bar f_2 , S^{<3,<3} \right)
\end{aligned}
\] 
and note that in the nonlinear terms, the order of differentiation of $f_1$ in the first 
line is $3$ in the nonlinear part while it is $2$ in the linear part on the second line. This 
behaviour has to be excluded.

 However, we have improved the system 
from \eqref{e:formalsys}, since the equations for $z \chi^2$ and 
for $z^2 \chi^3$ do not have this problem. We can thus use our crucial assumptions,
namely that
\begin{equation} \label{e:crucial}
\Phi_{1,2}' (Q + \Phi_{1,1}) + \Phi_{1,1}' \Phi_{1,2} = 0. \end{equation}
 Under this assumption, \eqref{eqn:realrelS1} implies that 
$S_{1,2} = -\bar S_{2,1}$, $S_{1,3} = -\bar S_{3,1}$, $S_{3,2} = -\bar S_{2,3}$, and we can 
replace the equations for these terms with their conjugate equations, therefore eliminating the derivatives of too high order.
Indeed, among the previous equations, consider each pair of equations of the form $L_{p,q}=\frac{S_{p,q}}{2i}+\tilde{\mathcal T}_{pq}$ and $(*) L_{q,p}=\frac{S_{q,p}}{2i}+\tilde{\mathcal T}_{qp}$. Assume that $\tilde{\mathcal T}_{qp}$ involves higher derivatives than $\tilde{\mathcal T}_{pq}$. Since $\bar S_{pq}=-S_{qp}$, we have 
$$
\tilde{\mathcal T}_{qp}= L_{q,p}-\frac{S_{q,p}}{2i}=L_{q,p}+\frac{\bar S_{p,q}}{2i}=L_{q,p}-\bar L_{p,q}+\overline{\tilde{\mathcal T}}_{pq}.
$$
Hence, we can replace equation (*) by $\bar L_{p,q}=\frac{S_{q,p}}{2i}+\overline{\tilde{\mathcal T}}_{pq}$, lowering thereby the order of the differentials invloved. Therefore, we obtain a system of the form
\[ 
\begin{aligned}
&z \chi \qquad & -\psi' Q +  Q(z,\bar f_1 ) +  Q (f_1 ,\chi) &= 
\frac{S_{1,1}}{2i} + \tilde{\mathcal T}_{1,1} \left(h^{\leq 1}, f_1, \bar f_1 \right)\\
&z^2 \chi \qquad & - 2i Q( z, \bar f_0' Q) + Q(f_2, \chi) & = \frac{S_{2,1}}{2i} + \overline{\tilde{\mathcal T}}_{1,2} \left({\bar h}^{\leq 1}, \bar f_1, f_1 , \bar S_{1,1} \right)
\\
&z^3 \chi \qquad & Q(f_3, \chi) & = \frac{S_{3,1}}{2i} + \overline{\tilde{\mathcal T}}_{1,3}  \left(\bar h^{\leq 1},\bar  f_1, \bar f_2, f_1 , {\bar S}^{<1, <3}\right)
\\
&z \chi^2 \qquad &  2 iQ(f_0' Q, \chi) + Q(z, \bar f_2)  & = \frac{S_{1,2}}{2i} + \tilde{\mathcal T}_{1,2} \left(h^{\leq 1}, f_1, \bar f_1 , S_{1,1} \right)
\\
&z \chi^3 \qquad & Q(z, \bar f_3) & = \frac{S_{1,3}}{2i} + \tilde{\mathcal T}_{1,3} \left(h^{\leq 1}, f_1, f_2, \bar f_1 , S^{<1, <3}\right)\\
&z^2 \chi^2 \qquad & - i \psi'' Q^2 +2 i  Q (f_1' Q , \chi)  & = \frac{S_{2,2}}{2i} + \tilde{\mathcal T}_{2,2} \left(h^{\leq 2} , f_1^{\leq 1} , f_2, \bar f_1, \bar f_2 , S^{<2,<2} \right)\\
&z^2 \chi^3 \qquad & -2 Q(f_0'' Q^2 , \chi)  & = \frac{S_{2,3}}{2i} + \tilde{\mathcal T}_{2,3} \left(h^{\leq 2} , f_1^{\leq 1}, f_2, \bar f_1 , \bar f_2 , S^{<2, <3}  \right)  \\
&z^3 \chi^2 \qquad & -2 Q (z, \bar f_0'' Q^2) &= \frac{S_{3,2}}{2i} + \overline{\tilde{\mathcal T}}_{3,2} \left(\bar h^{\leq 2} , \bar f_1^{\leq 1}, \bar f_2,  f_1 , f_2 , \bar S^{<2, <3} \right) \\
&z^3 \chi^3 \qquad & \frac{2}{3} \psi'''Q^3 - 2 Q(f_1'' Q^2, \chi)  & = \frac{S_{3,3}}{2i } + \tilde{\mathcal T}_{3,3} \left(h^{\leq 3}, f_1^{\leq 2}, f_2^{\leq 1}, \bar f_1, \bar f_2 , S^{<3,<3} \right)
\end{aligned}
\]
The equations for the $(2,1)$, the $(3,1)$ and the $(3,2)$ term
now depend nonlinearly on the conjugate $\bar S_{p,q}$, which 
we replace by their conjugates (i.e. the unbarred terms)
using the rules \eqref{eqn:realrelS1}. After that, 
we can  use the implicit function theorem in order to eliminate the dependence of the 
$\tilde{\mathcal T}_{p,q}$ on the $S_{p,q}$, obtaining the equivalent 
system of 
equations 
\[ \begin{aligned}
&z \chi \qquad & -\psi' Q +  Q(z,\bar f_1 ) +  Q (f_1 ,\chi) &= 
\frac{S_{1,1}}{2i} + {\mathcal T}_{1,1} \left(h^{\leq 1}, f_1, \bar f_1 \right)\\
&z^2 \chi \qquad & - 2i Q( z, \bar f_0' Q) + Q(f_2, \chi) & = \frac{S_{2,1}}{2i} + {{\mathcal T}}_{1,3} \left(h^{\leq 1}, \bar f_1, f_1 \right)
\\
&z^3 \chi \qquad & Q(f_3, \chi) & = \frac{S_{3,1}}{2i} + {\mathcal T}_{3,1} \left(h^{\leq 1}, \bar f_1, \bar f_2,  f_1 \right)
\\
&z \chi^2 \qquad &  2 iQ(f_0' Q, \chi) + Q(z, \bar f_2)  & = \frac{S_{1,2}}{2i} + {\mathcal T}_{1,2} \left(h^{\leq 1}, f_1, \bar f_1 , \right)
\\
&z \chi^3 \qquad & Q(z, \bar f_3) & = \frac{S_{1,3}}{2i} + {\mathcal T}_{1,3} \left(h^{\leq 1} , f_1, f_2, \bar f_1 \right)\\
&z^2 \chi^2 \qquad & - i \psi'' Q^2 +2 i  Q (f_1' Q , \chi)  & = \frac{S_{2,2}}{2i} + {\mathcal T}_{2,2} \left(h^{\leq 2} , f_1^{\leq 1}, f_2, \bar f_1, \bar f_2 \right)\\
&z^2 \chi^3 \qquad & -2 Q(f_0'' Q^2 , \chi)  & = \frac{S_{2,3}}{2i} + {\mathcal T}_{2,3} \left(h^{\leq 1}, f_1^{\leq 1}, f_2, \bar f_1 , \bar f_2 \right)  \\
&z^3 \chi^2 \qquad & -2 Q(z,\bar f_0'' Q^2)  &= \frac{S_{3,2}}{2i} + {{\mathcal T}}_{2,3} \left(h^{\leq 2}, f_1^{\leq 1}, f_2, \bar f_1 , \bar f_2 \right) \\
&z^3 \chi^3 \qquad & \frac{2}{3} \psi'''Q^3 - 2 Q(f_1'' Q^2, \chi)  & = \frac{S_{3,3}}{2i } + {\mathcal T}_{3,3} \left(h^{\leq 3}, f_1^{\leq 2}, f_2^{\leq 1}, \bar f_1, \bar f_2 \right)
\end{aligned}
\]
We use this system and substitute it (and its 
appropriate derivatives)
 into \eqref{eqn:simplerelphiS2} in order to obtain equations for the $\Phi_{p,q}$, leading to  
\begin{equation}\label{e:completesys1}\begin{aligned}
&z \chi \qquad & -\psi' Q +  Q(z,\bar f_1 ) +  Q (f_1 ,\chi) &= 
\Phi_{1,1} + {\mathcal T}_{1,1} \left(h^{\leq 1}, f_1, \bar f_1 \right)\\
&z^2 \chi \qquad & - 2i Q( z, \bar f_0' Q) + Q(f_2, \chi) & = \Phi_{2,1} + \overline{{\mathcal T}}_{1,2} \left(\bar h^{\leq 1}, \bar f_1, f_1 \right)
\\
&z^3 \chi \qquad & Q(f_3, \chi) & = \Phi_{3,1} + \overline{\mathcal T}_{1,3} \left(\bar h^{\leq 1}, \bar f_1, \bar f_2,  f_1 \right)
\\
&z \chi^2 \qquad &  2 iQ(f_0' Q, \chi) + Q(z, \bar f_2)  & = \Phi_{1,2} + {\mathcal T}_{1,2} \left(h^{\leq 1}, f_1, \bar f_1 , \right)
\\
&z \chi^3 \qquad & Q(z, \bar f_3) & = \Phi_{1,3} + {\mathcal T}_{1,3} \left(h^{\leq 1}, f_1, f_2, \bar f_1 \right)\\
&z^2 \chi^2 \qquad & i \left( Q (f_1' Q, \chi) - 
Q(z,\bar f_1' Q) \right) & = \Phi_{2,2} + \tilde{\mathcal S}_{2,2} \left(h^{\leq 2}, f_1^{\leq 1}, \bar f_1^{\leq 1}, f_2, \bar f_2 \right)\\
&z^2 \chi^3 \qquad & - i Q (z, \bar f_2' Q)  & = \Phi_{2,3} + \tilde{\mathcal S}_{2,3} \left(h^{\leq 2}, f_1^{\leq 1}, f_2, \bar f_1 , \bar f_2 \right)  \\
&z^3 \chi^2 \qquad &  i Q(f_2' Q,\chi) &= \Phi_{3,2} + \overline{\tilde{\mathcal S}}_{3,2} \left(\bar h^{\leq 2},  \bar f_1^{\leq 1},\bar  f_2,  f_1 ,  f_2 \right) \\
&z^3 \chi^3 \qquad &  \frac{1}{6} \psi'''Q^3 - \frac{1}{2}\left( Q(f_1'' Q^2, \chi) +Q (z, \bar f_1'' Q^2) \right)   & = \Phi_{3,3} + \tilde{\mathcal S}_{3,3} \left(h^{\leq 3}, f_1^{\leq 2}, f_2^{\leq 1}, \bar f_1^{\leq 2}, \bar f_2^{\leq 1} \right)
\end{aligned}
\end{equation}     
This system is now ``well graded'' so that we can expose it as a system of PDEs which allows
for the application of the big denominator theorem. However, we first single out 
the equations for  $z^2 \chi$, $z^3 \chi$, $z\chi^2$, $z\chi^3$:
\begin{equation}\label{e:implicit1}
\begin{aligned}
&z^2 \chi \qquad & - 2i Q( z, \bar f_0' Q) + Q(f_2, \chi) & = \Phi_{2,1} + \overline{{\mathcal T}}_{1,2} \left(\bar h^{\leq 1}, \bar f_1, f_1 \right)
\\
&z^3 \chi \qquad & Q(f_3, \chi) & = \Phi_{3,1} + \overline{\mathcal T}_{1,3} \left(\bar h^{\leq 1}, \bar f_1, \bar f_2,  f_1 \right)\\
&z \chi^2 \qquad &  2 iQ(f_0' Q, \chi) + Q(z, \bar f_2)  & = \Phi_{1,2} + {\mathcal T}_{1,2} \left(h^{\leq 1}, f_1, \bar f_1 , \right)
\\
&z \chi^3 \qquad & Q(z, \bar f_3) & = \Phi_{1,3} + {\mathcal T}_{1,3} \left(h^{\leq 1}, f_1, f_2, \bar f_1 \right)\\
\end{aligned}
\end{equation}
Applying the adjoint operator $\mathcal{K}^*$ to 
the system \eqref{e:implicit1} and 
using the normalization conditions \eqref{e:normalizep1} for 
the $(1,p)$ and $(p,1)$-terms for $p=2,3$  transforms them 
into a system of implicit equations for  $f_2$ and $f_3$ in term of $h^{\leq 1}$, $f_1$ and their conjugates:
\begin{equation}\label{e:implicit2}
\begin{aligned}
&z^2 \chi \qquad &\mathcal{K}^* \mathcal{K} f_2 & =  \mathcal{K}^*( 2i Q( z, \bar f_0' Q)+ \overline{{\mathcal T}}_{1,2}) 
\\
&z^3 \chi \qquad & \mathcal{K}^* \mathcal{K} f_3 & = \mathcal{K}^* \overline{\mathcal T}_{1,3} \\
\end{aligned}
\end{equation}
By the fact that $\mathcal{K}^* \mathcal{K}$ is invertible (on the image of $\mathcal{K}^*$, where the right hand side lies), we 
can solve this equation for $f_2$ and $f_3$ and 
substitute the result into the ``remaining'' equations to
obtain the following system: 
\begin{equation}\label{e:completesys2}\begin{aligned}
&z \chi \qquad & -\psi' Q +  Q(z,\bar f_1 ) +  Q (f_1 ,\chi) &= 
\Phi_{1,1} + {\mathcal T}_{1,1} \left(h^{\leq 1}, f_1, \bar f_1 \right) \\ 
&z^2 \chi^2 \qquad & i \left( Q (f_1' Q, \chi) - 
Q(z,\bar f_1' Q) \right) & = \Phi_{2,2} + {\mathcal S}_{2,2} \left(h^{\leq 2}, f_1^{\leq 1}, \bar f_1^{\leq 1}, f_2, \bar f_2 \right)\\
&z^3 \chi^3 \qquad &  \frac{1}{6} \psi'''Q^3 - \frac{1}{2}\left( Q(f_1'' Q^2, \chi) +Q (z, \bar f_1'' Q^2) \right)   & = \Phi_{3,3} + {\mathcal S}_{3,3} \left(h^{\leq 3}, f_1^{\leq 2}, f_2^{\leq 1}, \bar f_1^{\leq 2}, \bar f_2^{\leq 1} \right) \\
&z^3 \chi^2 \qquad &  -2 Q (z, \bar f_0 '' Q^2) &= \Phi_{3,2} - i \Phi_{2,1}' Q + \overline{{\mathcal S}}_{3,2}\left(\bar h^{\leq 2},  \bar f_1^{\leq 1},\bar  f_2,  f_1 ,  f_2 \right) 
\end{aligned}
\end{equation}
     
While coupled in the nonlinear parts, the linear parts of the equations corresponding to the 
diagonal terms of type $(1,1)$, $(2,2)$, and $(3,3)$ on the one hand and of the off-diagonal terms of type  $(3,2)$ (we drop from now on the conjugate term $(2,3)$) 
on the other hand are {\em decoupled}, the diagonal terms only 
depending on $f_1$ and $\psi$, the off-diagonal terms on $f_0$ and their derivatives.

We thus obtain the linear operator $\mathcal{L}$ already introduced in 
\cref{sec:a_full_formal_normal_form_proof_of_cref},  if we rewrite everything in terms 
of our operators $\Delta$, $\mathcal{K}$ and $\bar{\mathcal{K}}$ (see section \ref{sec:the_normalization_conditions}), 
 \begin{equation}\label{e:completesys3a}\begin{aligned}
&z \chi \qquad & -\Delta \psi +  \bar{\mathcal{K}}  \bar f_1  +  \mathcal{K} f_1 &= 
 \Phi_{1,1} + {\mathcal T}_{1,1}\\
&z^2 \chi^2 \qquad & i \left( \mathcal{K} \Delta f_1 - \bar{\mathcal{K}}  \Delta \bar f_1 \right) & =   \Phi_{2,2} + {\mathcal S}_{2,2}\\
&z^3 \chi^3 \qquad &  \frac{1}{6} \Delta^3 \psi - \frac{1}{2}\left( \mathcal{K} \Delta^2 f_1 + \bar{\mathcal{K}} \Delta^2 \bar f_1 \right)   & = \Phi_{3,3} + {\mathcal S}_{3,3} \\
\end{aligned}
\end{equation}     
The equation determining $f_0$ can be rewritten as 
\begin{equation}\label{e:completesys3b}\begin{aligned}
  -2 \bar{\mathcal{K}} \Delta^2 \bar f_0 &= \Phi_{3,2} - i \Delta \Phi_{2,1} + \overline{{\mathcal S}}_{3,2}
\end{aligned} \end{equation}
Let us stress that even though the linear terms here are the same 
as in \cref{sec:a_full_formal_normal_form_proof_of_cref}, the nonlinear terms are 
not the same as we had in that section, and an elimination of the derivatives of ``bad 
order'' like we did here is only possible under some restriction. 

However, with this in mind, we can completely proceed as in the proof of \cref{thm:main2}: 
we first project the equations on the normal form space $\mathcal{N}^{\rm off} \times \mathcal{N}^d$, and
obtain an equation of the form
\begin{equation}
\label{e:complete4}
\begin{aligned}
-2 \bar{\mathcal{K}} \Delta^2 \bar f_0 &= \pi_0 \overline{{\mathcal S}}_{3,2} \\
 -\Delta \psi +  \bar{\mathcal{K}}  \bar f_1  +  \mathcal{K} f_1 &= 
 \pi_1 {\mathcal T}_{1,1}\\
 i \left( \mathcal{K}  f_1 - \bar{\mathcal{K}}  \bar f_1 \right) & =  \pi_2 {\mathcal S}_{2,2}\\
 \frac{1}{6} \Delta^3 \psi - \frac{1}{2}\left( \mathcal{K} \Delta^2 f_1 + \bar{\mathcal{K}} \Delta^2 \bar f_1 \right)   & = \pi_3 {\mathcal S}_{3,3}
\end{aligned}
\end{equation}
We now ``homogenize'' the degree of differentials of these 
equations again, obtaining a system of the form 
\begin{equation}
\label{e:complete5}
\begin{aligned}
-2 \bar{\mathcal{K}} \Delta^3 \bar f_0 &=  \mathcal{F}_{3,2} \left( h^{\leq 3}, f_1^{\leq 2}, \bar f_1^{\leq 2} \right) \\
 -\Delta^3 \psi +  \bar{\mathcal{K}} \Delta^2  \bar f_1  +  \mathcal{K} \Delta^2 f_1 &= 
 \mathcal{F}_{1,1}\left( h^{\leq 3}, f_1^{\leq 2}, \bar f_1^{\leq 2} \right)\\
 i \left( \mathcal{K} \Delta^2 f_1 - \bar{\mathcal{K}} \Delta^2 \bar f_1 \right) & =  { \mathcal{F}}_{2,2} \left( h^{\leq 3}, f_1^{\leq 2}, \bar f_1^{\leq 2} \right)\\
 \frac{1}{6} \Delta^3 \psi - \frac{1}{2}\left( \mathcal{K} \Delta^2 f_1 + \bar{\mathcal{K}} \Delta^2 \bar f_1 \right)   & = \pi_3 {\mathcal S}_{3,3} \left( h^{\leq 3}, f_1^{\leq 2}, \bar f_1^{\leq 2} \right)
\end{aligned}
\end{equation}
Next, we substitute  $f_0$, $\real g_0$, and $f_1$ 
with $\tilde f_0 =f_0 - j^3 f_0$, $\real \tilde \psi = \psi - j^3 \psi$, and $\tilde f_1 = f_1 - j^2 f_1$ and obtain
\begin{equation}
\label{e:complete5b}
\begin{aligned}
-2 \bar{\mathcal{K}} \Delta^3 \bar{\tilde f}_0 &=  \tilde{\mathcal{F}}_{3,2} \left( \tilde h^{\leq 3}, \tilde f_1^{\leq 2}, \bar{\tilde{ f}}_1^{\leq 2} \right) \\
 -\Delta^3 \tilde \psi +  \bar{\mathcal{K}} \Delta^2  \bar{\tilde{ f_1}}  +  \mathcal{K} \Delta^2 \tilde f_1 &= 
 \tilde{\mathcal{F}}_{1,1} \left( \tilde h^{\leq 3}, \tilde f_1^{\leq 2}, \bar{\tilde{ f}}_1^{\leq 2} \right)\\
 i \left( \mathcal{K} \Delta^2 \tilde  f_1 - \bar{\mathcal{K}} \Delta^2 \bar{\tilde{ f_1}} \right) & =  \tilde{ \mathcal{F}}_{2,2} \left( \tilde h^{\leq 3}, \tilde f_1^{\leq 2}, \bar{\tilde{ f}}_1^{\leq 2} \right)\\
 \frac{1}{6} \Delta^3 \psi - \frac{1}{2}\left( \mathcal{K} \Delta^2 \tilde f_1 + \bar{\mathcal{K}} \Delta^2 \bar{\tilde{ f_1}} \right)   & = \tilde{ \mathcal{F}}_{3,3} \left( \tilde h^{\leq 3}, \tilde f_1^{\leq 2}, \bar{\tilde{ f}}_1^{\leq 2} \right).\end{aligned}
\end{equation}

 We can now apply the Big Denominator theorem \ref{main-thm-BD}
 to this system, just as we did in the proof of \cref{thm:main2}. The setup is the same, with $\real (\tilde{g}_0)$ now replaced
 by $(\psi, f_0)$, and the 
details are completely analogous to the details carried out in the 
proof of \cref{thm:main2} and therefore left to the reader.

\section{On the Chern-Moser normal form}\label{sec:chernmoser}
As we have already pointed out above, our normal form necessarily cannot agree with the 
normal form of Chern-Moser in the case $d=1$ (which we assume from now on). 
The reason is that we do not 
have  a choice of which normal form space to use for the diagonal terms-the operator
associated to all diagonal terms is injective, and we need to use its full adjoint. In the 
Chern-Moser case, the equation for the $(1,1)$-term, (with our notations from above)
$$ 
\Phi_{1,1} = \Delta \psi - {\mathcal K} f_1  - \bar{\mathcal K} \bar f_1 + \dots,
$$
is rather special, because {\em the operator $ f_1 \mapsto \real{\mathcal{K} f_1}$ is 
surjective}. (One can check that the weaker condition $\image \Delta \subset \image \real 
\mathcal{K}$ happens if and only if $d=1$).

 This means that if we look at the
 normal form condition for the $(1,p)$-terms, which 
just becomes $\Phi_{1,p} = 0$ (because $\mathcal{K}$ is surjective, $\mathcal{K}^*$ is
injective, and hence $\Phi_{1,p} = 0$ if and only if $\mathcal{K}^* \Phi_{1,p} = 0$), we 
can naturally also use it for the $(1,1)$-term and just request that 
$\Phi_{1,1} = 0$. A tricky point is 
that even though $\real \mathcal{K}$ is surjective (as  a map
on $\fpstwo{\mathcal{H}_1}{u})$), it is not injective.   By 
considering the polar decomposition $z + f_1 (z,u) = U(u)(I + R(u))  z$ with 
$U$ unitary with respect to $Q$, i.e. $Q(U(u)z,\overline{U}(u)\bar z) = Q(z, \bar z)$, the 
equation for the $(1,1)$-term becomes {\em an implicit equation for $R$} in terms of all
the other variables, because
\[ \begin{aligned} Q(z + f_1(z,u),\bar z + \bar f_1 (z,u)) &= Q(U(u)(I + R(u))  z, \overline{U}(u)(I + R(u))  \bar z) \\
&= Q(z, \bar z ) 
+ 2 \real Q (R(u) z, \bar z) + Q(R(u)z, R(u) \bar z).
\end{aligned} \]. We can then use the implicit function 
theorem to solve the $(1,1)$, $(2,1)$, and $(3,1)$-equations under the 
requirement $\Phi_{1,1} = \Phi_{2,1} = \Phi_{3,1} = 0$ jointly for $R$, $f_2$, and
 $f_3$  in terms of $U$ and $\real g_0$ and substitute 
the result back in all the other equations as we did before. If we follow this
procedure and go through with the rest of the arguments 
following \eqref{e:implicit1} with the 
appropriate changes, we obtain the Chern-Moser normal form; one just has to 
note that $u {\rm tr} \varphi  = \Delta^* \varphi$.

\appendix

\section{Computations}\label{comput}

We recall that $\Phi_{p,0}=\Phi_{0,q}=0$. 
Therefore, $(Q+\Phi)^l$ contains no terms $(p,q)$ with $p<l$ or $q<l$.
As a consequence, we have
\begin{eqnarray}
\re{g}_{p,0} &= & 0\label{5p0}\\
\re{g}_{p,1} &= & iD_ug_{p-2}(u)\Phi_{2,1}+iD_ug_{p-1}(u)\Phi_{1,1}\label{5p1}\\
\re{g}_{2,2} &= & iD_ug_0(u)\Phi_{2,2}+  iD_ug_1(u)\Phi_{1,2}\nonumber\\
&& +\frac{1}{2}D_u^2g_0(u)(2\Phi_{1,1}Q+\Phi_{1,1}^2)\label{522}\\
\re{g}_{3,3} &= & iD_ug_0(u)\Phi_{3,3}+  iD_ug_1(u)\Phi_{2,3}+ iD_ug_2(u)\Phi_{1,3}\nonumber\\
&& +\frac{1}{2}D_u^2g_0(u)(2\Phi_{2,2}Q+\{\Phi^2\}_{3,3})+\frac{1}{2}D_u^2g_1(u)(2\Phi_{1,2}Q+\{\Phi^2\}_{2,3})\nonumber\\
&& -\frac{i}{6}D_u^3g_0(u)(3\Phi_{1,1}^2Q+\Phi_{1,1}^3+3\Phi_{1,1}Q^2)\label{533}\\
\re{g}_{3,2} &= & iD_ug_0(u)\Phi_{3,2}+  iD_ug_1(u)\Phi_{2,2}+ iD_ug_2(u)\Phi_{1,2}\nonumber\\
&& +\frac{1}{2}D_u^2g_0(u)(2\Phi_{2,1}Q+\{\Phi^2\}_{3,2})+\frac{1}{2}D_u^2g_1(u)(2\Phi_{1,1}Q+\{\Phi^2\}_{2,2})\label{532}\\
\re{g}_{3,1} &= & iD_ug_0(u)\Phi_{3,1}+  iD_ug_1(u)\Phi_{2,1}+ iD_ug_2(u)\Phi_{1,1}\label{531}
\end{eqnarray}
To obtain $\bar g_{\geq 3}(z,u-iQ)-\bar g_{\geq 3}(z,u-iQ-i\Phi)$, we just use the previous result and substitute $g_k$ in $\bar g_k$ and $i$ by $-i$.
We have, using essentially the same computations~:
\begin{eqnarray}
\re{Qf}_{p,1} &= & \re{Qf}_{p,0}= 0\label{6p1}\\
\re{Qf}_{2,2} &= & Q(iD_uf_0(u)\Phi_{2,1}+  iD_uf_1(u)\Phi_{1,1},\bar C\bar z)\label{622}\\
\re{Qf}_{3,3} &= & Q(iD_uf_0(u)\Phi_{3,2}+  iD_uf_1(u)\Phi_{2,2}+ iD_uf_2(u)\Phi_{1,2}, \bar C\bar z)\nonumber\\
&& +\frac{1}{2}Q(D_u^2f_0(u)(2\Phi_{2,1}Q+\{\Phi^2\}_{3,2})+\frac{1}{2}D_u^2f_1(u)(2\Phi_{1,1}Q+\{\Phi^2\}_{2,2}),  \bar C\bar z)\label{633}\\
\re{Qf}_{3,2} &= & Q(iD_uf_0(u)\Phi_{3,1}+  iD_uf_1(u)\Phi_{2,1}+ iD_uf_2(u)\Phi_{1,1},\bar C\bar z)\label{632}
\end{eqnarray}

We have 
$$
Q\left(f_{\geq 2},\bar f_{\geq 2}\right)=\sum_{k,l\geq 0}\frac{i^{k+l}(-1)^l}{k!l!}Q\left(D^k_uf_{\geq 2}(z,u)(Q+\Phi)^k , D^l_u\bar f_{\geq 2}(\bar z,u)(Q+\Phi)^l \right).
$$
The function $D^k_uf_{j'}(z,u)(Q+\Phi)^k$ (resp. $D^l_u\bar f_{j}(\bar z,u)(Q+\Phi)^l$) has only terms $(p,q)$ with $p\geq j'+k$ and $q\geq k$ (rep. $p\geq l$ and $q\geq l+j$). Hence, the function 
$Q(D^k_uf_{j'}(z,u)(Q+\Phi)^k,D^l_u\bar f_{j}(z,u)(Q+\Phi)^l)$ contains only terms $(p,q)$ with $p\geq j'+k+l$ and 
$q\geq j+k+l$.
we have
\begin{eqnarray}
Q\left(f_{\geq 2},\bar f_{\geq 2}\right)_{p,0} &= & Q(f_p,\bar f_0)\label{Qffp0}\\
Q\left(f_{\geq 2},\bar f_{\geq 2}\right)_{p,1} &= & Q(f_p,\bar f_1)+iQ(Df_{p-1}(Q+\Phi_{1,1}), \bar f_0)-iQ(f_{p-1}, D_u\bar f_0(Q+\Phi_{1,1}))\label{Qffp1}\\
Q\left(f_{\geq 2},\bar f_{\geq 2}\right)_{2,2} &= & Q(f_2,\bar f_2)+iQ(Df_1(Q+\Phi_{1,1}),\bar f_1)-iQ(f_1, D\bar f_1(Q+\Phi_{1,1}))\nonumber\\
&&-\frac{1}{2}\left(Q(f_0, D^2_u\bar f_0(Q+\Phi_{1,1})^2)+Q(D_u^2f_0(Q+\Phi_{1,1})^2, \bar f_0)\right)\nonumber\\
&&-Q\left(D_uf_0(u)(Q+\Phi_{1,1}),D_u\bar f_0(u)(Q+\Phi_{1,1})\right)\label{Qff22}\\
Q\left(f_{\geq 2},\bar f_{\geq 2}\right)_{3,3} &= & Q(f_3,\bar f_3)+
iQ(Df_0\Phi_{3,1}+Df_1\Phi_{2,1}+Df_2(Q+\Phi_{1,1}), \bar f_2)\nonumber\\
&&-iQ(f_2,D_u\bar f_0\Phi_{1,3}+D_u\bar f_1\Phi_{1,2}+D_u\bar f_2(Q+\Phi_{1,1}))\nonumber\\
&&+Q(i(D_uf_0\Phi_{3,2}+D_uf_2(Q+\Phi_{1,1})-\frac{1}{2}(D_u^2f_0(Q+\Phi_{1,1})\Phi_{2,1}+ D^2_uf_1(Q+\Phi_{1,1})^2)),\bar f_1)\nonumber\\
&&+Q(f_1, -i(D_u\bar f_0\Phi_{2,3}+D_u\bar f_2(Q+\Phi_{1,1})-\frac{1}{2}(D_u^2\bar f_0(Q+\Phi_{1,1})\Phi_{1,2}+ D^2\bar f_1(Q+\Phi_{1,1})^2)))\nonumber\\
&&Q(\frac{-i}{3}D^3_uf_0(Q+\Phi_{1,1})^3+\frac{-1}{2}(D^2_uf_0(Q+\Phi_{1,1})\Phi_{2,2}+D^2_uf_1(Q,\Phi_{1,1})\Phi_{1,2}), \bar f_0)\nonumber\\
&&Q(f_0, \frac{i}{3}D^3_u\bar f_0(Q+\Phi_{1,1})^3+\frac{-1}{2}(D^2_u\bar f_0(Q+\Phi_{1,1})\Phi_{2,2}+D^2_u\bar f_1(Q,\Phi_{1,1})\Phi_{2,1}))\nonumber\\
&&+ Q(-i(D_uf_0\Phi_{3,3}+D_uf_1\Phi_{2,3}+D_uf_2\Phi_{2,3}), \bar f_0)\nonumber\\
&&+ Q(f_0, i(D_u\bar f_0\Phi_{3,3}+D_u\bar f_1\Phi_{3,2}+D_u\bar f_2\Phi_{3,2}))\nonumber\\
&& \frac{-i}{2}Q\left(D_uf_0(u)(Q+\Phi_{1,1}), D^2_u\bar f_0(u)(Q+\Phi_{1,1})^2\right)\nonumber\\
&&+ \frac{i}{2}Q\left(D^2_uf_0(u)(Q+\Phi_{1,1})^2, D_u\bar f_0(u)(Q+\Phi_{1,1})\right)\nonumber\\
&& +Q(iD_uf_1(z,u)(Q+\Phi_{1,1}), D_u\bar f_1(\bar z,u)(Q+\Phi_{1,1}))\label{Qff33}\\
Q\left(f_{\geq 2},\bar f_{\geq 2}\right)_{3,2} &= & Q(f_3,\bar f_2)-i Q(f_2, D_u\bar f_1(Q+\Phi_{1,1})+D_u\bar f_0\Phi_{1,2})\nonumber\\
&& -iQ(f_1, D_u\bar f_0(Q+\Phi_{1,1})^2+D_u\bar f_1\Phi_{2,1})-iQ(f_0, D_u\bar f_1\Phi_{3,1}+D_u\bar f_0\Phi_{3,2})\nonumber\\
&&+Q\left(D_uf_1(z,u)(Q+\Phi_{1,1}), D_u\bar f_0(u)(Q+\Phi_{1,1})\right)-\frac{1}{2}Q(D^2_uf_1(Q+\Phi_{1,1})^2,\bar f_0)\label{Qff32}
\end{eqnarray}

We have 
\begin{equation}
\tilde \Phi_{\geq 3}\left(f,\bar f, \frac{1}{2}(g+\bar g)\right)- \tilde \Phi_{\geq 3}\left(Cz,\bar C\bar z, su\right)=\sum_{\substack{|\alpha|+|\beta|+|\gamma|=k\\k\geq 1}}\frac{1}{\alpha!\beta!\gamma!}\frac{\partial^k \tilde\Phi_{\geq 3}}{\partial z^{\alpha}\bar z^{\beta}u^{\gamma}} (Cz,\bar C\bar z, su)f_{\geq 2}^{\alpha}\bar f_{\geq 2}^{\beta} \left(\frac{1}{2}(g_{\geq 3}+\bar g_{\geq 3})\right)^{\gamma}\label{bigsum}
\end{equation}
where $\alpha,\beta\in \Bbb N^n$ and $\gamma\in \Bbb N^d$.
Hence, the $(p,q)$ term of $\tilde \Phi_{\geq 3}\left(f,\bar f, \frac{1}{2}(g+\bar g)\right)- \tilde \Phi_{\geq 3}\left(Cz,\bar C\bar z, su\right)$ is a sum of terms of the form
\begin{equation}\label{genericterm}
\left\{\frac{\partial^k \tilde\Phi_{\geq 3}}{\partial z^{\alpha}\partial\bar z^{\beta}\partial u^{\gamma}}(Cz,\bar C\bar z, su)\right\}_{p_1,q_1} \{f_{\geq 2}^{\alpha}\}_{p_2,q_2}\{\bar f_{\geq 2}^{\beta}\}_{p_3,q_3} \left\{\left(\frac{1}{2}(g_{\geq 3}+\bar g_{\geq 3})\right)^{\gamma}\right\}_{p_4,q_4}
\end{equation}
with $\sum_{i=1}^4p_i=p, \sum_{i=1}^4q_i=q$. 

Let us first compute $\{f_{\geq 2}^{\alpha}\}_{p_2,q_2}$ with $p_2,q_2\leq 3$. In the following computations, $f,g$ are considered as vector valued functions except when computing $f^{\alpha}, (g+\bar g)^{\gamma}$ where $f,g$ are considered as scalar functions and $\alpha,\gamma$ as an integers.

In the sums below, the terms appear with some positive multiplicity that we do not write since we are only interested in a lower bound of vanishing order of the terms. Fron these computations, we easily obtain $\{\bar f_{\geq 2}^{\alpha}\}_{p_2,q_2}$ in the following way~: replace $f_k$ by $\bar f_k$ in formula defining $\{f^{\alpha}\}_{p,q}$ in order to obtain $\{\bar f^{\alpha}\}_{q,p}$. Furthermore, we have
$$
\left\{\frac{\partial^k \tilde\Phi_{\geq 3}}{\partial z^{\alpha}\partial\bar z^{\beta}\partial u^{\gamma}}\right\}_{p_1,q_1} = \frac{\partial^k \tilde\Phi_{p_1+|\alpha|,q_1+|\beta|}}{\partial z^{\alpha}\bar z^{\beta}u^{\gamma}}
$$
Let us set as notation 
$$
Re(g):=\frac{g+\bar g}{2}= \frac{g(z,u+i(Q(z,\bar z)+\Phi(z,\bar z,u)))+\bar g(\bar z,u-i(Q(z,\bar z)+\Phi(z,\bar z,u))}{2}.
$$

\section{Big denominators theorem for non-linear systems of PDEs}
\label{sec-PDEs}

In this section we recall one of the main results of article \cite{Stolovitch:2014tk} about local analytic solvability of some non-linear systems of PDEs that have the ``big denominators property''.

\subsection{The problem}

Let $r\in\Bbb N^*$ and let ${\bf m} = (m_1,\ldots,m_r)\in \Bbb N^r$ be a fixed multiindex. Let us denote $\ank $ (resp. $\left(\ank \right)_{>d}$, $\widehat{\ank}$, $\left(\ank \right)^{(i)}$ ) the space of $k$-tuples of germs at $0\in \mathbb R^n$ (or $\Bbb C^n$) of analytic functions (resp. vanishing at order $d$ at the origin, formal power series maps, homogeneous polynomials of degree $i$) of $n$ variables. Let us set
$$
{\mathcal F}_{r,\bf m}^{\geq 0} := \left(\mathbb A_n\right)_{\geq m_{1}} \times \left(\mathbb A_n\right)_{\geq m_{2}}
\times \cdots \times \left(\mathbb A_n\right)_{\geq m_{r}}
$$
Given $F=(F_1,...,F_r)\in \mathcal F_{r,\bf m}^{\geq 0}$
and $x\in (\mathbb R^n,0)$, let us denote 
$$
j_x^{{\bf m}}F  :=\left(j^{m_{{1}}}_x F_1, \ \cdots , \ j^{m_{{r}}}_x F_r\right), \ \
J^{{\bf m}}\mathcal F_{r,\bf m}^{\geq 0} :=\left\{\left(x, \ j^{{\bf m}}_x F\right), \ \ x\in (\mathbb R^n,0), \ F\in  \mathcal F_{r,\bf m}^{\geq 0}\right\}.
$$
%
\begin{definition}
\label{def-diff-map}
A map $\mathcal T: \cF_{r,\bf m}^{\geq 0}\to \mathbb A_n^s$ is a {\bf differential analytic map of order ${\bf m}$} at the point
$0\in \mathbb A_n^k$ if there exists an analytic map germ $$W: \left(J^{{\bf m}}\mathcal F_{r,\bf m}^{\geq 0}, 0\right) \to \mathbb R^s$$
such that $\mathcal T(F)(x) = W(x, j^{\bf m}_xF)$ for any $x\in \mathbb R^n$ close to $0$ and any function germ $F\in \cF_{r,\bf m}^{\geq 0}$ such that $j^m_0F$ is close to $0$.
\end{definition}

Denote by $$v = \left(x_1,...,x_n, u_{j, \alpha }\right), \ \
  1\leq j\leq r, \ \alpha = (\alpha _1,..., \alpha _n)\in\Bbb N^n, \ \vert \alpha \vert \le m_j$$
   the local coordinates in $J^{\bf m}\mathbb A_n^r$, where $u_{j, \alpha }$ corresponds to the partial
   derivative $\partial^{|\alpha|} /\partial x_1^{\alpha _1}\cdots \partial x_n^{\alpha _n}$ of the $j$-th component
   of a vector function $F\in \mathbb A_n^r$. As usual, we have set $|\alpha|=\alpha_1+\cdots+\alpha_n$.
%

\begin{definition}\label{def-nonlin-bd}
Let $q$ be a nonnegative integer.
Let $\mathcal T: \cF_{r,\bf m}^{\geq 0}\to \mathbb A_n^s$ be a map. 
\begin{itemize} 
\item We shall say that it {\bf increases the order at the origin} (resp. strictly) by $q$ if  
for all $(F,G)\in (\cF_{r,\bf m}^{\geq 0})^2$ then 
$$
\text{ord}_0\left({\mathcal T}(F)-{\mathcal T}(G)\right)\geq \text{ord}_0(F-G)+q,
$$
(resp. $>$ instead of $\geq$).
\item Assume that ${\cT}$ is an analytic differential map of order $\bf m$
defined by a map germ $W:  \left(J^{{\bf m}}\mathcal F_{r,\bf m}^{\geq 0}, 0\right) \to \mathbb R^s$ as in Definition \ref{def-diff-map}. 
We shall say that it is {\bf regular}
 if, for any formal map $F=(F_1,\ldots, F_r)\in \widehat\cF_{r,\bf m}^{\geq 0}$, then
$$
\text{ord}_0\left(\frac{\partial W_i}{\partial u_{j,\alpha}}(x,\partial F)\right)\geq p_{j,|\alpha|},
$$
where 
\begin{equation}
p_{j,|\alpha|}=\max(0, |\alpha|+q+1-m_j)\label{p-reg}
\end{equation}
We have set $\partial F:=\left(\frac{\partial^{|\alpha|} F_i}{\partial x^{\alpha}},\;1\leq i\leq r,\,0\leq |\alpha|\leq m_i  \right)$.
\end{itemize}
\end{definition}

Let us consider linear maps~:
\begin{enumerate}
\label{triple}
\item 
$$
\mathcal S: \cF_{r,\bf m}^{\geq 0}\ \to \ \mathbb A_n^s,
$$
that increases the order by $q$ 
and is homogenous, i.e ${\mathcal S}\left(\cF_{r,\bf m}^{(i)}\right)\subset (\mathbb A_n^s)^{(q+i)}$.
\item 
$$
\pi : \mathbb A_n^s \to Image \hskip .05cm (\mathcal S) \subset \mathbb A_n^s
$$
is a projection onto $ Image \hskip .05cm (\mathcal S)$.
\end{enumerate}

Let us consider a differential analytic map of order $\bf m$, $\mathcal T:\cF_{r,\bf m}^{\geq 0}\ \to \ \mathbb A_n^s$. 

We consider the equation
\begin{equation}
\label{eq-main-bd}
\mathcal S(F) = \pi \left( \mathcal T(F)\right)
\end{equation}
In \cite{Stolovitch:2014tk}, we gave a sufficient condition on the triple $\left(\mathcal S, \mathcal T, \pi \right)$ under which equation (\ref{eq-main-bd}) has a solution $F\in \cF_{r,\bf m}^{\geq 0}$; this condition is called the ``Big Denominators property" of the triple $(\mathcal S,\mathcal T,\pi )$ defined below.

\subsection{Big denominators. Main theorem}

Now we can define the big denominators property of the triple $(\mathcal S, \mathcal T, \pi )$ in equation (\ref{eq-main-bd}).

\begin{definition}
\label{def-big-denominators}
The triple of maps $(\mathcal S, \mathcal T, \pi )$ of form (\ref{triple})
has 
 {\bf big denominators property of order $\bf m$}
if there exists an nonnegative integer $q$ such that the following holds:

\begin{enumerate}

\item \ $\mathcal T$ is an regular analytic differential map of order $\bf m$ that strictly increases the order by $q$ 
 and $j^{q-1}_0\mathcal T(0)=0$, i.e.  $\mathcal T(F)(x) = W(x, j^{\bf m}_xF)$ for any $x\in \mathbb R^n$ close to $0$ and any function germ $F\in \cF_{r,\bf m}^{\geq 0}$ such that $j^m_0F$ is close to $0$ and $ord_0(W(x, 0))\geq q$ 
.

\item $\cS: \cF_{r,\bf m}^{\geq 0}\ \to \ \mathbb A_n^s$ is linear 
, increases the order by $q$ 
 and is homogenous, i.e. ${\mathcal S}\left(\cF_{r,\bf m}^{(i)}\right)\subset (\mathbb A_n^s)^{(q+i)}$.

\item the linear map $\pi : \mathbb A_n^s \to Image \hskip .05cm (\mathcal S) \subset \mathbb A_n^s$ is a projection.

\item \ the map $\mathcal S$ admits right-inverse $\mathcal S^{-1}: Image (S)\to \mathbb A_n^r$ such that
the composition $\mathcal S^{-1}\circ \pi $ satisfies:

there exists $C>0$ such that for any $G\in \mathbb A_n^s$ of order $> q$, one has for all $1\leq j\leq r$, and all integer $i$, 
\begin{equation}
\label{BD}
\left\|\left(\cS_j^{-1}\circ \pi (G)\right)^{(i+m_j)}\right\|\leq C\frac{\|G^{(i+q)}\|}{(i+m_j+q)\cdots (i+q+1)}.
\end{equation}
where ${\cal S}_i^{-1}$ denotes the $i$th component of ${\cal S}^{-1}$, $1\leq i\leq r$ 
.
\end{enumerate}
\end{definition}
\begin{remark}
Let $i\ge 0$ and let $F = (F_1,..., F_k)\in \left(\mathbb A_n^k\right)^{(i)}$. Let
$F_j = \sum F_{j, \alpha }x^\alpha $ where the sum is taken over all $j=1,...,k$ and all
multiindexes $\alpha = (\alpha _1,..., \alpha _n)$ such that $\vert \alpha \vert = \alpha _1+\cdots + \alpha _n = i$.
The norm $\vert \vert F \vert \vert$ used in \re{BD} is either 
\begin{itemize}
\item $$\vert \vert F_j \vert \vert = \sum _{\vert \alpha \vert = i}\vert F_{j,\alpha }\vert, \ \ \vert \vert F \vert \vert = max \left(\vert \vert F_1\vert \vert , \cdots ,  \vert \vert F_k\vert \vert \right).$$
\item or the modified Fisher-Belitskii norm
$$\vert \vert F_j \vert \vert^2 = \sum _{\vert \alpha \vert = i}\frac{\alpha!}{|\alpha|!}\vert F_{j,\alpha }\vert^2, \ \ \vert \vert F \vert \vert^2 = \vert \vert F_1\vert \vert^2+\cdots + \vert \vert F_k\vert \vert^2.$$

\end{itemize}
\end{remark}
\begin{remark}
In practice, for each $i$, there is a decomposition into direct sums $\cF_{r,\bf m}^{(i)}=L_i\oplus K_i$ with ${\cal S}_{|L_i}$ is a bijection onto its range. The chosen right inverse is then the one with zero component along $K_i$. For instance, the case of the modified Fisher-Belitskii norm, $K_i:=\ker {\cal S}_i^*$ is the natural one, where ${\cal S}_i^*$ denotes the adjoint of ${\cal S}_i$ w.r.t. the scalar product.
\end{remark}
\begin{theorem}\cite{Stolovitch:2014tk}[theorem 7]
\label{main-thm-BD}
Let us consider a system of analytic non-linear pde's such as equation (\ref{eq-main-bd})~:
\begin{equation}
\mathcal S(F) = \pi \left( W(x, j^{\bf m}_xF)\right).\label{bd-equ2solve}
\end{equation}
If the triple $(\mathcal S, \mathcal T,\pi )$  has big denominators property of order $\bf m$,
according to definition \ref{def-big-denominators}, then the equation has an analytic solution $F\in \cF_{r,\bf m}^{\geq 0}$ 
.
\end{theorem}
\begin{remark}
The precise statement of \cite{Stolovitch:2014tk} holds for $F\in \cF_{r,\bf m}^{> 0}$ and where the order of $W(x,0)$ at the origin is greater than $q$. The shift by 1  (i.e $F\in \cF_{r,\bf m}^{\geq 0}$ and where the order of $W(x,0)$ at the origin is greater or equal to $q$) of the above statement, doesn't affect its proof.
\end{remark}

\subsection{Application}
In this section we shall devise the strictly increasing condition in more detail. We look for
a formal solution $F^{\geq 0}=\sum_{i\geq 0}F^{(i)}$ to \re{bd-equ2solve}. As above, $F^{(i)}$ stands for $(F_1^{(m_1+i)},\ldots, F_r^{(m_r+i)})$. We define 
$$
\mathcal S(F^{(i+1)}):=\left[\pi  W\left(x, j^{\bf m}_x\sum_{j\geq 0}^{i} F^{(j)}\right)\right]^{(i+q+1)}.
$$
Here $[G]^{(i)}$ denotes the homogenous part of degree $i$ of $G$ in the Taylor expansion at the origin. Therefore $F:=\sum_{i\geq} F^{(i)}$ is a solution of \re{bd-equ2solve} if 
\begin{equation}
\ord_0\left(W\left(x, j^{\bf m}_x\sum_{j\geq 0} F^{(j)}\right)-W\left(x, j^{\bf m}_x\sum_{j\geq 0}^{i} F^{(j)}\right)\right)>i+q+1.\label{strcond}
\end{equation}
Indeed, we would have 
\begin{eqnarray*}
\mathcal S\left(\sum_{i\geq 0} F^{(i)}\right)& =& \sum_{i\geq 0}\left[\pi  W\left(x, j^{\bf m}_x\sum_{j\geq 0}^{i} F^{(j)}\right)\right]^{(i+q+1)} \\
&=& \sum_{i\geq 0}\left[\pi  W\left(x, j^{\bf m}_x\sum_{j\geq 0} F^{(j)}\right)\right]^{(i+q+1)}= \pi  W\left(x, j^{\bf m}_xF\right)
\end{eqnarray*}
We emphasize that condition \re{strcond} just means that $W$ strictly increases the order by $q$ as defined in Definition \ref{def-nonlin-bd}. Let us look closer to that condition. Let us 
denote $F^{\leq i}:=\sum_{j\geq 0}^{i} F^{(j)}$ and $F^{> i}:=\sum_{j>i} F^{(j)}$. Let us Taylor expand $W(x, j^{\bf m}_xF)$ at $F^{\leq i}$. We thus have
\begin{eqnarray*}
W(x, j^{\bf m}_xF)-W(x, j^{\bf m}_xF^{\leq i})&= &\sum \frac{\partial W}{\partial u_{j,\alpha}}((x, j^{\bf m}_xF^{\leq i}))\frac{\partial^{|\alpha|} F_j^{> i}}{\partial x^{\alpha}}\\
&&+\frac{1}{2} \sum \frac{\partial W}{\partial u_{j,\alpha}\partial u_{j',\alpha'}}((x, j^{\bf m}_xF^{\leq i}))\frac{\partial^{|\alpha|} F_j^{> i}}{\partial x^{\alpha}}\frac{\partial^{|\alpha'|} F_{j'}^{> i}}{\partial x^{\alpha'}}+\cdots
\end{eqnarray*}
We recall that $\ord_0F_j^{> i}>m_j+i$ and when considering a coordinate $u_{j,\alpha}$, we have $|\alpha|\leq m_j$. Hence, we have
$$
\ord_0\frac{\partial^{|\alpha|} F_j^{> i}}{\partial x^{\alpha}} >  m_j+i-|\alpha|.
$$
In order that the first derivative part of this Taylor expansion satisfies \re{strcond}, it is sufficient that 
$$
\ord_0\frac{\partial W}{\partial u_{j,\alpha}}((x, j^{\bf m}_xF^{\leq i}))\geq |\alpha|-m_j+q+1.
$$ 
This is nothing but the {\it regularity condition} as defined in Definition  \re{def-nonlin-bd}. Let us consider the other terms in the Taylor expansion. We have, for instance, 
$$
\ord_0\frac{\partial^{|\alpha|} F_j^{> i}}{\partial x^{\alpha}}\frac{\partial^{|\alpha'|} F_{j'}^{> i}}{\partial x^{\alpha'}}\geq m_j+i+1-|\alpha|+m_{j'}+i+1-|\alpha'|
$$

If $i+1> q$, then not only the second but also any higher order derivative part of this Taylor expansion satisfies \re{strcond}.
\begin{corollary}\label{q=0}
If $q=0$ and if the system is regular, then it strictly increases the order by 0.
\end{corollary}

\bibliographystyle{alpha}
\bibliography{chern-moser}
\end{document}